\documentclass[10pt]{amsart}

\usepackage{amssymb,amsmath}
\usepackage{amsthm}
\usepackage{esint}
\usepackage{xcolor}
\usepackage{comment}
\usepackage{mathrsfs}
\usepackage{graphicx}

\title{
Hele-Shaw flow as a singular limit of a Keller-Segel system with nonlinear diffusion
}

\author{Antoine Mellet}
\address{Department of Mathematics, University of  Maryland, College Park, MD}
\email{mellet@umd.edu}
\thanks{
Partially supported by NSF Grant DMS-2009236.
}

\def\R{\mathbb R}
\def\eps{\varepsilon}

\def\pa{\partial}
\def\na{\nabla}
\def\div{\mathrm{div}\,}

\def\BV{\mathrm{BV}}

\newcommand{\J}{\mathscr{J}}
\newcommand{\G}{\mathscr{G}}
\newcommand{\F}{\mathscr{F}}

\def\H{\mathcal H}

\numberwithin{equation}{section}

\newtheorem{theorem}{Theorem}[section]
\newtheorem{theorem*}{Theorem}
\newtheorem{remark}[theorem]{Remark}
\newtheorem{lemma}[theorem]{Lemma}

\newtheorem{proposition}[theorem]{Proposition}
\newtheorem{corollary}[theorem]{Corollary}

\begin{document}

\maketitle

\begin{abstract}
We study a singular limit of the classical parabolic-elliptic Patlak-Keller-Segel (PKS) model for chemotaxis with non linear diffusion. The main result is the $\Gamma$ convergence of the corresponding energy functional toward the perimeter functional.
Following recent work on this topic, we then prove that under an energy convergence assumption, the solution of the PKS model converges to a solution of the Hele-Shaw free boundary problem with surface tension, which describes the evolution of the interface separating regions with high density from those with low density.
This result complements a recent work by the author with I. Kim and Y. Wu, in which the same free boundary problem is  derived from the incompressible PKS model (which includes a density constraint $\rho\leq 1$ and a pressure term): 
It shows that the incompressibility constraint is not necessary to observe phase separation   and surface tension phenomena.
\end{abstract}

\medskip
\medskip

\medskip
\medskip


\medskip
\medskip
\medskip
\medskip

\section{Introduction}
We consider the classical parabolic-elliptic Keller-Segel model for chemotaxis with nonlinear diffusion (see \cite{Keller_Segel,Patlak,JL92,HP}):
\begin{equation}\label{eq:KS2}
\begin{cases}
\pa_t \rho - \Delta\rho^m +\beta \div(\rho \na \phi)=0 ,\\
- \eps^2  \Delta \phi  =  \rho - \sigma\phi  .
\end{cases}
\end{equation}
In this  model 
$\rho(t,x)$ is the density function of a population of bacteria (or other type of cells, amoebae or even animals)
which diffuses and is advected by a drift $\na \phi$ toward the region of higher concentration of a chemoattractant. 
This chemical is being secreted by the bacteria themselves and its 
 concentration $\phi(t,x)$ solves an elliptic equation. 
The parameter  $\beta$ is the cell sensitivity and $\sigma>0$ represents the degradation of the chemical. 
Importantly, the diffusion is nonlinear and we assume that $m>2$ (we will also consider more general diffusion term - see Assumptions {\bf (H1)}-{\bf (H3)} below). 
Finally, we note that the chemical's diffusivity, denoted here by $\eps^2$ is destined to go to zero.

\medskip

The system \eqref{eq:KS2} is set on  an open bounded subset $\Omega \subset \R^d$ with smooth boundary in dimension $d\geq 2$. 

\medskip
  
In this model, the diffusion of the bacteria competes with the attracting  potential $\phi$ which causes the cells to aggregate. 
This competition has been well studied in the linear diffusion case $m=1$.
In particular, it is well-known that, for some initial conditions, the concentration of the bacteria will lead to 
finite time blow-up of the density ($\limsup_{t\to t^*} \| \rho(t)\|_\infty =\infty$  - see e.g. \cite{JL92}, \cite{HV}).
These blow-ups can be interpreted as the formation of saturated regions with very high value of the density function, 
but from a biological point of view, arbitrary large density of bacteria should not be allowed.
Various modifications of the classical Keller-Segel model have been proposed to prevent blow-up. There are  two main strategies: Enhancing the diffusion  or inhibiting the attraction when the density is large.
Our model \eqref{eq:KS2} with $m>1$ is an example of the former strategies and has been studied by several authors \cite{CC06,K05,sugiyama} (we refer to \cite{HP01} for models with inhibited attraction).

In some  recent papers \cite{KMW1,KMW2} (in collaboration with I. Kim and Y. Wu), we took a slightly different approach:
We considered a similar chemotaxis system, with linear diffusion ($m=1$) but with a hard incompressibility constraint $\rho\leq 1$. This constraint  requires the introduction of a pressure term in the density equation, as is common in fluid mechanics - see \cite{KMW1} for details.
This incompressible model can actually be obtained as the limit $m\to\infty$ of \eqref{eq:KS2} - but this is not our purpose here and $m$ will be fixed in the present paper.
The main result of  \cite{KMW1,KMW2} is the fact that the constraint $\rho\leq1$ leads to the formation of saturated patches $\{\rho(t)=1\}$ (regions of maximal densities) and   that the evolution of these patches  can be described by classical Hele-Shaw free boundary problems.
In particular, 
 we characterized (in \cite{KMW2}) the limiting behavior of the density when $\eps\ll1$ and $\beta\sim \eps^{-1}$:
In that regime, the attractive effect of the potential $\phi$ together with the constraint $\rho\leq1$ leads to the convergence of the density toward a characteristic function $\chi_{E(t)}$ (a phenomena known as phase separation). Furthermore, the localization of the potential $\phi$  when $\eps\ll1$ leads to surface tension phenomena. 
More precisely, under some assumptions on the  convergence of the energy, we proved that the evolution of $E(t)$ could be described by the Hele-Shaw problem with surface tension (see \eqref{eq:HS} below).
\medskip

The goal of the present paper is to show that this convergence occurs for the solutions of \eqref{eq:KS2} as well, without the hard constraint $\rho\leq 1$, but with nonlinear diffusion $m>2$. In other words,
when the diffusion is strong enough for large values of $\rho$ the balance between the attractive  potential and the diffusion leads, in a time scale of order $\eps^{-1}$ to phase separation and surface tension phenomena.

\medskip

Let us now write the rescaled boundary value problem that we will be working with in this paper: We are interested in the asymptotic behavior of the solutions of \eqref{eq:KS} 
when $\eps\ll1$ at time scale $\eps^{-1}$.
Rescaling the time variable accordingly and introducing  the relevant boundary (null flux) and initial conditions, we are  led to the following system of equations:
\begin{equation}\label{eq:weak}
\begin{cases}
\eps \pa_t \rho  -\Delta \rho^m  +\beta  \div( \rho \na \phi)=0 ,\quad & \mbox{ in } \Omega\times(0,\infty),\\
(- \na\rho^m +\beta \rho \na \phi)\cdot n = 0,\quad & \mbox{ on }\pa \Omega\times(0,\infty)\\
\rho(x,0) = \rho_{in}(x)\quad & \mbox{ in } \Omega
\end{cases}
\end{equation}
with $\phi(\cdot,t)$ solution of 
\begin{equation}\label{eq:phi0}
\begin{cases}
\sigma  \phi -\eps^2\Delta  \phi = \rho  &  \mbox{ in } \Omega\\
 \nabla \phi \cdot n = 0 & \mbox{ on } \pa\Omega.
\end{cases}
\end{equation}

\begin{remark}
The scaling of \eqref{eq:weak}-\eqref{eq:phi0} can also be interpreted as follows: 
Let $\bar \rho(\bar x,\bar t)$  be solution of the system  \eqref{eq:weak}-\eqref{eq:phi0} with $\eps=1$ (here $\bar x$, $\bar t$ denotes the microscopic variables) and a  large number of bacteria: 
$$\int_{\R^n}\bar \rho(\bar x,\bar t) \, d\bar x = \int_{\R^n}\bar \rho_{in}(\bar x) \, d\bar x  \gg 1.$$
We then introduce $\eps \ll1$ such that $ \int_{\R^n}\bar \rho_{in}(\bar x) \, d\bar x = \eps^{-n}$ and 
 rescale the time and space variables
 $$ x=\eps \bar x, \qquad t=\eps^{3} \bar t .$$
The function $\rho^\eps(x,t)=\bar \rho(\bar x,\bar t)$ then solves \eqref{eq:weak}-\eqref{eq:phi0}  and satisfies $\int_{\R^n} \rho^\eps(x,t) \, dx =1$.
\end{remark}

\medskip

The first part of our result concerns the phenomena of phase separation.
We note that the natural energy for \eqref{eq:weak}-\eqref{eq:phi0} is
$$
E_\eps(\rho) = 
\begin{cases}
\displaystyle \frac{1}{\eps} \int_\Omega \frac1{m-1} \rho(x)^m   - \frac \beta 2 \rho(x) \phi(x)\, dx & \mbox{ if } f(\rho)\in L^1(\Omega), \rho\geq 0 \\
\infty & \mbox{otherwise}
\end{cases}
$$
(where $\phi$ is given by \eqref{eq:phi0}).
We will prove (see Theorem \ref{thm:Gamma} for a precise statement) that there is a constant $A_\eps$ such that the functional $\rho\mapsto E_\eps(\rho)+A_\eps$ (which is also an energy for the system) 
$\Gamma$-converges when $\eps\to 0$ to 
$$
E_0(\rho) = 
\begin{cases}
\gamma\rho_c  P(E) & \mbox{ if } \rho = \rho_c \chi_E \in \BV(\Omega)\\
\infty & \mbox{otherwise}
\end{cases}
$$
for some constants $\gamma$ and $\rho_c$. This implies that
if the initial condition $\rho_{in}^\eps$ is such that $  E_\eps(\rho^\eps_{in})+A_\eps$ remains bounded (which holds in particular if the initial condition is of the form $\rho_c \chi_{E_{in}}$), then the solution 
 $\rho^\eps$ of \eqref{eq:weak}-\eqref{eq:phi0} converges to a characteristic function $\rho_c \chi_{E(t)}$.
 \smallskip
 
The second part of our result (see Theorem \ref{thm:conv1}) describes the evolution of this set $E(t)$ and shows that under a classical assumption on the convergence of the energy \eqref{eq:EA}, this evolution can be  described by the following Hele-Shaw free boundary problem with surface tension:
\begin{equation}
\label{eq:HS}
\begin{cases}
\Delta p = 0 & \mbox{ in } E(t),\qquad \na p\cdot n =0 \quad\mbox{ on } \pa\Omega \cap E(t)\\
p = \gamma_0  \kappa & \mbox{ on } \pa E(t) \cap\Omega \\
V = -\na p \cdot \nu & \mbox{ on } \pa E(t)\cap\Omega.
\end{cases}
\end{equation}
where $\kappa(x,t)$ denotes the mean-curvature of $\pa E(t)$ (with the convention that $\kappa\geq 0$ when $E$ is convex) and $V$ denotes the normal velocity of the interface $\pa E(t)$.
This problem must be supplemented with a contact angle condition at the triple junction $\pa E(t)\cap\pa\Omega$.
Since $\phi$ satisfies Neumann boundary condition, the contact angle condition  derived in \cite{KMW2} reads
\begin{equation} \label{eq:CA}
\cos \theta =0
\end{equation}
which simply states that the moving interface $\pa E(t)$ must be orthogonal to the fixed boundary $\pa\Omega$ ($\theta$ denotes the angle formed by $\pa E(t)$ and $\pa\Omega$).
\smallskip

We recall that \eqref{eq:weak}-\eqref{eq:phi0} (respectively \eqref{eq:HS}) is the gradient flow for the energy $E_\eps$ (respectively $E_0$) with respect to the Wasserstein metric on the manifold of probability measures. This fact will not be used in this paper, but it provides a heuristic justification for the convergence result Theorem \ref{thm:conv1} 
and it underscores the importance of the $\Gamma$-convergence result, Theorem \ref{thm:Gamma}.

\medskip

The Hele-Shaw free boundary problem with surface tension   \eqref{eq:HS}  is a  very classical model, classically introduced to model the 
motion of the interface separating two immiscible fluids in an unbounded
Hele-Shaw cell. 
It has been derived in various frameworks, in particular as the limit of the discrete time approximation of a gradient flow for the perimeter functional (see for instance \cite{OttoL,CL})
and as the sharp interface limit of a
Cahn-Hilliard equation with degenerate mobility (see   \cite{Glasner,Laux1}).
The derivation of \eqref{eq:HS} in \cite{KMW2} was the first to make a direct connection between this model and chemotaxis phenomena, but it should be noted that the energy functional used in that paper shares some similarities with the functional used in \cite{JKM,LO} to derive a related two-phase free boundary problems.

\medskip

We point out that the density constrained model that we considered in \cite{KMW1,KMW2} is the natural limit of our (unconstrained) model \eqref{eq:KS2} when $m\to \infty$. One would thus expect to recover the result of 
\eqref{eq:KS2} in the regime $m\gg1$ and $\eps\ll1$.
However, and this is the main object of this paper, it turns out that it is not necessary to impose the hard constraint $\rho\leq 1$ or to take the limit $m\to\infty$ in order
to observe these phenomena of phase separation and surface tension since
we derive the Hele-Shaw free boundary problem with surface tension \eqref{eq:HS} in the limit $\eps\ll1$ when $m$ is fixed satisfying $m>2$.

\medskip

To conclude this introduction, we mention that our result will be proved for general non-linear diffusion, not necessarily given by some power law. More precisely, we consider the system:
\begin{equation}\label{eq:KS}
\begin{cases}
\pa_t \rho + \div (\rho v) = 0 \quad & \mbox{ in } \Omega \times (0,\infty) \\
v = \eps^{-1}\na [- f'(\rho) +\beta    \phi] & \mbox{ in } \Omega \times (0,\infty)\\
- \eps^2  \Delta \phi  =  \rho - \sigma\phi & \mbox{ in } \Omega \times (0,\infty) \\
 \rho v \cdot n = 0 \, , \quad \na \phi\cdot n =0 \qquad &  \mbox{ on }\pa \Omega \times (0,\infty)
 \end{cases}
\end{equation}
which corresponds to  \eqref{eq:KS2} when
\begin{equation}\label{eq:mum}
 f(\rho) = \frac{1}{m-1}\rho^{m}.
 \end{equation}
One way to understand the conditions that will be required on the nonlinearity $f(\rho)$, is to note that \eqref{eq:phi0} gives $\phi= \frac 1 \sigma \rho + \mathcal O(\eps^2)$ and so the evolution of $\rho$ when $\eps^2\ll 0$ is approximated by
\begin{equation}\label{eq:KS'}
\eps \pa_t \rho =  \div(\rho \na f'(\rho)) - \frac{\beta}{\sigma} \div(\rho \na \rho)=  \div\left (\rho\left[f''(\rho) -  \frac{\beta}{\sigma}\right] \na \rho\right)  .
\end{equation}
If $f''(\rho) -  \frac{\beta}{\sigma}\geq 0$, this is a well-behaved diffusion equation, but  we are interested here in models for which  $f''(\rho) -  \frac{\beta}{\sigma}$ is positive for large $\rho$ but negative for small $\rho$. In such cases,  \eqref{eq:KS'} is  a forward-backward diffusion equation, which typically leads to phase separation. 
We will list later some general assumptions on $f$ for our result to hold but we can already see that this applies when $f$ is given by \eqref{eq:mum}  with $m>2$.

While \eqref{eq:KS'} explains why phase separation occurs in our limit, it does not lead to surface tension phenomena. These can be  understood by looking at the next order approximation in  \eqref{eq:phi0}:
 $\phi =  \frac{1}{\sigma}\rho + \frac{\eps^2 }{\sigma} \Delta \phi  \sim \frac{1}{\sigma}\rho + \frac{\eps^2 }{\sigma^2} \Delta \rho +\mathcal O(\eps^4)$.
With this in mind, we see that  the system \eqref{eq:weak} is formally approximated by the following Cahn-Hilliard equation with degenerate mobility coefficient:
\begin{align}
 \pa_t \rho  
 & =-\eps^{-1} \div\left(\rho \na \left(  - f'(\rho) +\frac{\beta}{\sigma}  \rho +\frac{\beta}{\sigma^2} \eps^2 \Delta \rho \right)\right)\nonumber \\
&= -\div \left(\rho \na \left( \eps \frac{\beta}{\sigma^2} \Delta \rho - \frac 1 \eps W_0'(\rho)\right)\right)\label{eq:CH}
\end{align}
where the function $W_0(\rho)  = f(\rho) - \frac{\beta}{2\sigma} \rho^2 $ is a double well potential  when $f$ is given by \eqref{eq:mum}  with $m>2$ thanks to the constraint $\rho\geq 0$ (see Figure \ref{fig:1}) - with wells at $0$ and $\rho_c>0$.

\begin{figure}[]
 	 		\includegraphics[width=.35\textwidth]{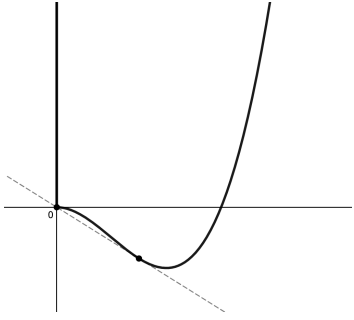}
			\vspace*{-40pt}
			
 	 		\caption{The double-well potential $W_0(\rho)=f(\rho) - \frac{\beta}{2\sigma} \rho^2$ when $f(\rho)=\rho^3$}
 	 		\label{fig:1}
 	 	\end{figure}

The singular limit $\eps\to 0$ 
for the classical Cahn-Hilliard equation with constant mobility was formally derived by Pego in \cite{Pego} and justified rigorously, for example, in \cite{ABC,Chen,Le}. This limit leads to phase separation and the free boundary separating the two phases evolves according to a two-phase Hele-Shaw free boundary problem with surface tension (also called  Mullins-Sekerka). In the case of degenerate mobility, which we are considering here, a similar formal analysis was performed by Glasner in \cite{Glasner} and leads to the one-phase Hele-Shaw problem \eqref{eq:HS}.
This convergence was justified rigorously  more recently by Kroemer and Laux in \cite{Laux1}. 
Of course, this approximation \eqref{eq:CH} is only valid if $\rho$ is smooth, so it cannot be used to justify the limit $\eps\to0$, but the result of \cite{Laux1} is similar to our result in many ways (in particular it also requires an energy convergence assumption).

Finally, we point out  that the limiting free boundary problem \eqref{eq:HS} is a one-phase problem because of  the degeneracy of the mobility coefficient $\rho$ at one of the two wells $\rho_-=0$. 
If we include in \eqref{eq:KS} a nondegenerate diffusion term $\Delta \rho$ as well, it adds a term $\rho \log\rho$ to the potential $W_0(\rho)$, which would then have two positive wells ($0<\rho_-<\rho_+$). While we do not pursue the corresponding  analysis here, the singular limit would lead in that case to the two-phase Hele-Shaw problem.

\medskip

\medskip

\medskip

\section{Notations and main results}

\subsection{The pressure $f'(\rho)$ and the potential $W(\rho)$}
When $f$ is given by \eqref{eq:mum} with $m>2$, the function 
$$h:= \rho \mapsto 
\frac 1 \rho \left( f(\rho) - \frac{\beta}{2\sigma} \rho^2\right)$$
has a unique minimum for $\rho\geq0$, attained when $\rho = \rho_c:=\left( \frac{\beta}{2\sigma} \right) ^{\frac{1}{m-2}} $.
We also denote 
$$a := - h(\rho_c) = \frac{m-2}{m-1} \left( \frac{\beta}{2\sigma}\right)   ^{\frac{1}{m-2}}$$
so that we have $h(\rho)\geq -a$ for all $\rho\geq 0$ with equality if and only if $\rho=\rho_c$.
\medskip

We now introduce the function
\begin{equation}\label{eq:W}
 W(\rho) := \rho(h(\rho)+a) = 
  f(\rho) - \frac{\beta}{2\sigma} \rho^2 + a\rho.
\end{equation}
This is a $C^2$ function on $[0,\infty)$ and the definition of $a$ implies that
\begin{equation}\label{eq:W0}
W(\rho)\geq 0 \mbox{ for all $\rho\geq 0\quad $ and  } \quad W(\rho)=0 \Leftrightarrow \rho = 0 \mbox{ or } \rho_c.
\end{equation}
Because of the natural constraint $\rho\geq 0$ in the problem, we can  thus view  $W$ as a double-well potential (we can also define $W$ in $\R$ by setting  $W(\rho)=+\infty$ for $\rho<0$).
Furthermore, we clearly have 
\begin{equation}\label{eq:Wqg}
W(\rho) \geq \nu \rho^2 \qquad \mbox{ for all } \rho\geq C.
\end{equation}

\medskip

As mentioned in the introduction, we do not need to assume that $f$ is a power law of the form \eqref{eq:mum} in order to prove our result.
All we need is to be able to define the double-well potential $W$ as in \eqref{eq:W}, satisfying
the properties listed above. 
This is the case if we assume that $f$ satisfies the following general conditions:
\medskip
\begin{itemize}
\item [{\bf (H1)}] The function 
$\rho \mapsto f(\rho)$ is a $C^2$ strictly convex function on $\R_+$ 
satisfying $f''(\rho)>0$ for all $\rho>0$ and 
$f''(\rho) \geq \delta$ for $\rho\geq \delta^{-1}$ (for some $\delta>0$).
\end{itemize}
\medskip

Since $f$ is defined up to a linear part, we can assume (without loss of generality) that
 $$ f(0)=f'(0)=0.$$

\begin{itemize}
\item [{\bf (H2)}] The function 
$h:= \rho \mapsto 
\frac 1 \rho \left( f(\rho) - \frac{\beta}{2\sigma} \rho^2\right)$
is bounded below on $\R_+$ and there exists a unique $\rho_c\in(0,\infty)$ such that 
\begin{equation}\label{eq:a0}
h(\rho_c) = \inf_{\rho\geq 0} h(\rho) = : - a. 
\end{equation}

\item[{\bf (H3)}] There exists $\nu > 0$ and $C\geq 0$ such that 
the function  $\rho \mapsto f(\rho)$ satisfies
\begin{equation}\label{eq:fquad3} 
f(\rho) \geq \left(\frac{\beta}{2\sigma} + \nu\right)   \rho^2 \qquad
 \mbox{ for all } \rho\geq C.
\end{equation}
\end{itemize}

Assumption {\bf (H1)} says that the pressure $\rho\mapsto f'(\rho) $ must be monotone increasing and grow at least linearly for large $\rho$. In particular it implies that $f(\rho)\geq \frac \delta  2 \rho^2$ for large $\rho$. 
Assumption {\bf (H2)} is also a condition on the quadratic growth of $f$, but it is more precise: 
The fact that $h$ is bounded below on $\R_+$ is immediate if, for example, we have $f(\rho) \geq \frac\beta {2\sigma} \rho^2$ for large enough $\rho$. This condition also implies that $\liminf_{\rho\to\infty} h(\rho)\geq 0=h(0)$ so that the minimum is reached in $[0,\infty)$.
If we have in addition that  $f(\rho) \leq \frac\beta {2\sigma }\rho^2$ for small $\rho$, then this minimum is reached in $(0,\infty)$.
The uniqueness of $\rho_c$ in {\bf (H2)} ensures that we have phase separation with two (and only two)  values of $\rho$ ($0$ and $\rho_c$). 
Assumption {\bf (H2)} implies that the potential $W$ defined by \eqref{eq:W} satisfies $W(\rho)\geq 0$ for all $\rho\geq 0$, and  {\bf (H3)} gives \eqref{eq:Wqg} which will be used in the proof of Theorem \ref{thm:conv1}.

\subsection{Gamma convergence of the energy functional}
As in \cite{KMW2}, the derivation of the Hele-Shaw free boundary problem in the limit $\eps\to0$ relies heavily on the properties of the energy functional associated to \eqref{eq:weak}-\eqref{eq:phi0}.
First, we recall that the solution $\phi$ of \eqref{eq:phi0} can be written as
$$ \phi^\eps(x,t) = \int_\Omega G_\eps(x,y) \rho(y,t)\, dy$$
for some kernel $G_\eps$ satisfying in particular $G_\eps(x,y)=G_\eps(y,x)$.
We then define
\begin{align*}
\J_\eps(\rho) 
& = \frac 1 \eps\left[  \int_\Omega f(\rho(x)) + a \rho(x)\, dx  - \frac \beta 2  \int_\Omega \int_\Omega G_\eps(x,y) \rho(x)  \rho(y)\, dy \, dx \right] \\
& = \frac 1 \eps   \int_\Omega f(\rho(x)) + a \rho(x)   - \frac \beta 2   \rho(x)\phi^\eps(x)  \, dx  .
\end{align*}
The system \eqref{eq:weak}-\eqref{eq:phi0} is a gradient flow for this energy with respect to the Wasserstein distance $W_2$ (the linear term $a\rho$ plays no role here since the mass is preserved, but it will be important later - the same is true of the scaling $\frac 1 \eps$).
This expression makes apparent the competition between the repulsive effect of diffusion (the $f$ term) and attractive effect of advection.
However, a simple but important computation shows that this energy can then also be written as follows:
\begin{equation}\label{eq:JW}
\J_\eps(\rho) = \frac 1 \eps \int_\Omega W(\rho)  \, dx +   \frac 1 \eps \int_\Omega\frac \beta {2\sigma}  (\rho-\sigma \phi^\eps)^2\, dx +\eps  \int_\Omega \frac\beta  2 |\na \phi^\eps|^2\, dx 
\end{equation}
with $W$ given by \eqref{eq:W}. 
Indeed, using \eqref{eq:phi0}, we can write
\begin{align*}
- \int_\Omega \rho \phi^\eps\, dx 
& = \frac 1 \sigma\int_\Omega -\rho^2 + (\rho-\sigma \phi^\eps)^2 + \sigma \phi(\rho-\sigma \phi^\eps) \, dx \\
& =\frac 1 \sigma \int_\Omega -\rho^2 + (\rho-\sigma \phi^\eps)^2 \, dx- \eps^2\int_\Omega \phi \Delta\phi^\eps \, dx \\
& = \frac 1 \sigma\int_\Omega -\rho^2 + (\rho-\sigma \phi^\eps)^2\, dx  + \eps^2 \int_\Omega|\na \phi^\eps|^2  \, dx .
\end{align*}
We deduce 
\begin{align*}
\J_\eps(\rho) 
& = \frac 1 \eps \int_\Omega f(\rho) + a\rho - \frac \beta {2\sigma} \rho^2  \, dx +  \frac 1 \eps \int_\Omega\frac \beta {2\sigma}  (\rho-\sigma \phi^\eps)^2\, dx +\eps  \int_\Omega \frac\beta  2 |\na \phi^\eps|^2\, dx 
\end{align*}
which is \eqref{eq:JW} (see \eqref{eq:W} for the definition of $W$).
\medskip

We recall (see \eqref{eq:W0}) that $W$ is a double-well potential. The functional  \eqref{eq:JW} is thus reminiscent of the classical Modica-Mortola functional, except that the double-well potential involves $\rho$ while the gradient term involves $\phi^\eps$. But the middle term controls the $L^2$ distance between $\rho$ and $\sigma \phi^\eps$. It is thus not unreasonable to think that this functional behaves, for small $\eps$, like the Modica-Mortola functional and $\Gamma$-converges to the perimeter functional. 
This will indeed be our first result. 
To state this, we first recall that the solution of \eqref{eq:KS} satisfies $\rho\geq 0$ and $\int_\Omega \rho\, dx = m_0$ (with $m_0$ the initial mass).
In what follows we take $m_0=1$ and define
\begin{equation}\label{eq:Jeps}
\J_\eps(\rho) = 
\begin{cases}
\displaystyle \frac 1 \eps \int_\Omega W(\rho)  \, dx +   \frac 1 \eps \int_\Omega\frac \beta {2\sigma}  (\rho-\sigma \phi^\eps)^2\, dx +\eps  \int_\Omega \frac\beta  2 |\na \phi^\eps|^2\, dx &  \mbox{ if } \rho\in L^2(\Omega), \;\rho\geq 0 \mbox{ and } \int_\Omega \rho=1\\
\infty & \mbox{otherwise in $L^1(\Omega)$}
\end{cases}
\end{equation}
(we could also remove the condition $\rho\geq 0$ and extend the definition of $W$ by setting $W(\rho) = +\infty$ for $\rho<0$).
The first result of the paper is:
\begin{theorem}\label{thm:Gamma}
Assume that $\Omega$ is an open subset of $\R^d$ with Lipschitz boundary and that $f$ satisfies
{\em ({\bf H1}), ({\bf H2})}.
Then 
the energy functional  $\J_\eps$ $\Gamma$-converges to $\J_0$ defined by
$$
\J_0(\rho)= \begin{cases}
\displaystyle \gamma  \int_\Omega |\na \rho|   & \mbox{ if } \rho \in \BV(\Omega;\{0,\rho_c\}) \mbox{ and } \int_\Omega \rho=1 \\
\infty & \mbox{ otherwise in $L^1(\Omega)$}
\end{cases}
$$
for some constant  $\gamma $ defined by \eqref{eq:sigma} below.
More precisely, we have 
\item[(i)] {\bf liminf property}: For all sequence $\rho^\eps \in L^1(\Omega)$ such that $\rho^\eps \to \rho$ in $L^1(\Omega)$, we have
$$\liminf_{\eps\to 0 } \J_\eps(\rho^\eps) \geq \J_0(\rho).$$
\item[(ii)] {\bf limsup property}: For all $\rho \in L^1(\Omega)$, there exists a sequence $\rho^\eps\in L^1(\Omega)$ such that $\rho^\eps\to\rho$ in  $L^1(\Omega)$ and 
$$\limsup_{\eps\to 0}  \J_\eps(\rho^\eps) \leq \J_0(\rho).$$
\end{theorem}
The definition of the constant $\gamma$ appearing in the definition of $\J_0$ requires the introduction of the function
$$
 g(s) : = \inf_{\rho\geq 0}  \left[ W(\rho) +  \frac {\beta}{ 2\sigma} (\rho-\sigma s)^2\right]
$$
which plays an important role in the proof.
We note that $g(s)\geq 0$ for all $s\geq 0$ and $g(0)=g(\rho_c/\sigma) =0$ (in fact $g$ is a double-well potential).
The constant $\gamma$ appearing in the definition of $\J_0$ is then given by
\begin{equation}\label{eq:sigma}
\gamma:=\frac 1 {\rho_c}\int_0^{\rho_c/\sigma } \sqrt{2\beta g(s)}\, ds.
\end{equation}

The proof of Theorem \ref{thm:Gamma} is presented in Section \ref{sec:Gamma}.
The key idea is to introduce the functional
$$ \G_\eps(\eta,\psi) =  \frac 1 \eps \int_\Omega W(\eta) + \frac1  2 (\eta-\psi)^2 +\eps \int_\Omega \frac1  2 |\na \psi|^2\, dx$$
for all functions $\eta\in L^2(\Omega) $ and $\psi\in H^1(\Omega)$ and notice (see Proposition \ref{prop:GJ}) that for a given $\rho$, the infimum of $\G_\eps(\rho,\psi)$ over all $\psi\in H^1(\Omega)$ is actually achieved when $\psi =\phi^\eps$ solution of \eqref{eq:phi0}.

\begin{remark}\label{rem:sup}
When $\rho(x) = \rho_c \chi_E \in  \BV(\Omega;\{0,\rho_c\})$, we have $W(\rho(x))=0$ a.e. and so
\eqref{eq:JW} reduces to
$$
\J_\eps(\rho_c \chi_E) =  \frac 1 \eps \int_\Omega\frac \beta {2\sigma}  (\rho-\sigma \phi)^2\, dx +\eps  \int_\Omega \frac\beta  2 |\na \phi|^2\, dx 
$$
This is the same functional that we studied in \cite{MW} (see \cite[Proposition 3.4]{MW}). 
It follows in particular from \cite[Proposition 2.3]{MW} and its proof that
$$ \J_\eps(\rho_c \chi_E) \leq C \int_\Omega |\na \rho| = C \rho_c P(E,\Omega)$$
for some constant $C$ independent of $\eps$ and 
$$ \lim_{\eps\to 0}  \J_\eps(\rho_c \chi_E) = \frac{\beta \rho_c }{4 \sigma^{3/2}}  \int_\Omega |D\rho|  = \frac{\beta \rho_c^2 }{4 \sigma^{3/2}} P(E,\Omega) .$$

The definition of $g(s)$ implies $g(s) \leq \frac{\beta}{2\sigma} \min\{ (\sigma s)^2, (\rho_c-\sigma s)^2\}$ (taking $\rho=0$ and $\rho=\rho_c$ in the infimum) which leads to
$$ \gamma \leq  \frac{\beta \rho_c}{4 \sigma^{3/2}}$$
with, in general, a strict inequality. We thus expect to have  $\lim_{\eps\to 0}  \J_\eps(\rho_c \chi_E) > \J_0(\rho_c \chi_E)$.
This remark shows that  we cannot simply take $\rho^\eps=\rho$ to prove the limsup property in Theorem \ref{thm:Gamma}. Instead, the proof of Theorem \ref{thm:Gamma}-(ii) will require the delicate construction of a recovery sequence $\rho^\eps$.
\end{remark}

\subsection{Weak solutions of \eqref{eq:weak}-\eqref{eq:phi0}}
The Patlak-Keller-Segel system of equations \eqref{eq:weak}-\eqref{eq:phi0} has been 
extensively studied. 
Many results 
concerning the existence of global in time solutions vs. finite time blow-up
have been obtained in the linear diffusion case $f(\rho) = \rho \log \rho$ with $\sigma=0$ (see for instance \cite{JL92,BCC12,CD14,BDP06}).
It is also well known that nonlinear diffusion can prevent blow-up and guarantee the existence of bounded solutions globally in time  (see \cite{sugiyama,K05,CC06}).
In dimension $2$, bounds on $\| \rho\|_{L^\infty(0,\infty;L^\infty(\Omega))}$  were established in \cite{CC06,K05} assuming that $\rho f''(\rho) \geq \kappa $ for large $\rho$ and when $\kappa\geq \kappa_*$ (which holds under our assumption {\bf (H1)}).
In  dimension $d\geq 2$, and when $f'(\rho) = \rho^{m-1}$,
solutions can be proved to be bounded $L^\infty(\Omega)$ globally in time without any restriction on the size of the initial data in the subcritical case $m> 2-\frac 2 d$, see \cite{sugiyama} (while finite time blow-up may occur for some initial data when $1<m<2-\frac 2 d$).
Similar bounds have been established in \cite{carrillo-wang} for very general nonlocal interaction potential with the same non-linear diffusion.

We will use the following  result:
\begin{theorem}\label{thm:existence}
 Let $\Omega$ be a bounded subset of $\R^d$ with $C^{1,\alpha}$ boundary 
 and assume that $\sigma,\beta>0$ and that $f(\rho)$ satisfies {\em ({\bf H1})}.
For any initial data  $\rho_{in } \in L^1\cap L^\infty(\Omega)$,  there exists a weak solution $\rho^\eps$ of  \eqref{eq:weak}-\eqref{eq:phi0} defined globally in time  satisfying
\item[(i)]  $\rho \in L^\infty(0,\infty;L^1 \cap L^\infty(\Omega))$, $\phi\in L^\infty(0,\infty;H^1(\Omega))$
\item[(ii)] $\rho$ is a solution of the continuity equation $\pa_t \rho + \div (\rho v) =0$ for some $v\in L^2((0,\infty)\times\Omega,d\rho)$, that is
\begin{equation}\label{eq:weak11}
\int_\Omega \rho_{in} (x) \zeta(x,0)\, dx + \int_0^\infty \int_\Omega \rho\, \pa_t\zeta + \rho v \cdot \na \zeta \, dx = 0 
\qquad \forall \zeta\in C^\infty_c([0,\infty)\times\overline \Omega).
\end{equation}
\item[(iii)] The flux $\rho v $ satisfies
\begin{equation}\label{eq:weak12}
\int_0^\infty \int_\Omega  \rho v\cdot \xi - \eps^{-1}\beta \rho\na \phi \cdot \xi - \eps^{-1} [ \rho f'(\rho ) -f(\rho)]\, \div\xi\, dx \, dt= 0
\end{equation}
for any vector field  $\xi \in C^\infty_c((0,\infty)\times\overline \Omega ; \R^d)$ such that $\xi \cdot n=0$ on $\pa\Omega$.
\item[(iv)] For all $t>0$, the function $\phi(\cdot,t)$ is the unique solution of \eqref{eq:phi0} in the following weak sense:
\begin{equation} \label{eq:phiweak}
\phi\in H^1(\Omega) , \quad \int_\Omega (\rho(\cdot,t)-\phi) \psi + \eps^2 \na \phi \cdot\na \psi\, dx=0 \qquad \forall \psi \in H^1(\Omega).
\end{equation}
\item[(v)] The following energy dissipation property holds:
\begin{equation}\label{eq:diss}
\J_\eps(\rho(t)) + \int_0^t \int_\Omega \rho |v|^2\, dx\, dt \leq \J_\eps(\rho_{in}) \qquad \forall t>0.
\end{equation}
\end{theorem}

Equality \eqref{eq:weak11} is the usual weak formulation for the continuity equation $\pa_t \rho + \div (\rho v)=0$ with Neumann boundary conditions and initial condition $\rho_{in}$. Equation \eqref{eq:weak12} is equivalent to the equality $\rho v = - \eps^{-1} \rho\na f'(\rho)+ \eps^{-1}\beta \rho \na \phi  $.

Since $\rho\in L^\infty(0,T; L^\infty(\Omega))$, 
Inequality \eqref{eq:diss}, implies 
that  $\rho v \in L^2(0,T; L^{2}(\Omega))$ and so the continuity equation gives $\rho \in C^{1/2}(0,T; H^{-1}(\Omega))$. In particular, it make sense to have the (pointwise in time) function $\rho(\cdot,t)$ in \eqref{eq:phiweak}.

\medskip

Some version of this result is proved for example in \cite{sugiyama} (when $\Omega=\R^d$ and $f'(\rho) = \rho^{m-1}$ with $m>2-\frac 2 d$) and \cite{CC06,K05} (when $\Omega$ subset of $\R^2$).
Weak solutions can also be constructed using a JKO scheme much as in \cite{KMW1} (where a similar system, but with the density constraint $\rho\leq1$, is considered) or \cite{Laux1} (where this gradient flow approach is developed for the degenerate Cahn-Hilliard  model, which is closely related to our model).
We also point out that Theorem 4 in  \cite{Carrillo-Santambrogio} shows that when $f''>0$ and $\rho_{in}\in L^\infty(\Omega)$ (with $\Omega$ convex bounded domain) our equation (with Dirichlet boundary condition for $\phi$) has a weak bounded solution for some time $T = (\beta \| \rho_{in}\|_{L^\infty} )^{-1}$.
Finally, we note that uniqueness of such weak solutions is proved in
 \cite{CaLiMa14}.

Importantly, we point out that all the $L^\infty$ bounds mentioned above are not expected to be uniform in $\eps$ and are therefore not useful for the purpose of this paper.
Only the natural bound in $L^\infty(0,\infty;L^1(\Omega))$ and the bound in $L^\infty(0,\infty;L^2(\Omega))$ provided by the energy inequality under assumption {\bf (H3)} will be used in the proof of Theorem \ref{thm:conv1} below.

\medskip

\subsection{Convergence results}
Our second main result is the following convergence theorem:
\begin{theorem}\label{thm:conv1}
Assume $f$ satisfies {\em ({\bf H1}), ({\bf H2})} and {\em ({\bf H3})}.
Given a sequence $\eps_n$ and initial data $\rho_{in}^{\eps_n}$ such that $\J_{\eps_n}(\rho_{in}^{\eps_n})\leq M$ and $\rho_{in}^{\eps_n}\to \rho_{in} = \rho_c \chi_{E_{in}} \in BV(\Omega;\{0,\rho_c\})$,  let $\rho^{\eps_n}(x,t)$ be the weak solution 
of \eqref{eq:weak}-\eqref{eq:phi0} (given by Theorem \ref{thm:existence}). Then the  followings hold:
\item[(i)] Along a subsequence, the density $\rho^{\eps_n}(x,t)$  converges strongly in $L^\infty((0,T);L^1(\Omega))$ to 
$$\rho(x,t) = \rho_c \chi_{E(t)} \in L^\infty((0,\infty);\BV(\Omega;\{0,\rho_c\})).$$ 
\item[(ii)]  There is a velocity function $v(x,t)$ such that $ t\mapsto \chi_{E(t)}$ satisfies the continuity equation 
\begin{equation}\label{eq:contweakHS}
\begin{cases}
\pa_t \chi_{E(t)} + \div(\chi_{E(t)}v) = 0 & \mbox{ in } \Omega \\
\chi_E v \cdot n =0 &  \mbox{ on } \pa\Omega \\
\chi_{E(0)}= \chi_{E_{in}} & \mbox{ in } \Omega 
\end{cases}
\end{equation}
 as well as the 
energy dissipation property
\begin{equation}\label{eq:energy2}
\J_0(\rho(t)) + \int_0^t \int_\Omega |v|^2 \rho \, dx\, dt \leq \liminf \J_{\eps_n}(\rho_{in}^{\eps_n}).
\end{equation}
\item[(iii)] The pressure defined by
$$p^{\eps_n}:= \frac 1 {\eps_n} \left[ \rho^{\eps_n} (f'(\rho^{\eps_n}) - f'(\rho_c) ) + \frac \beta \sigma \rho^{\eps_n} (\rho_c -\sigma \phi^{\eps_n})\right]$$ 
converges to a function $ p(x,t)$ weakly-$*$ in $L^2((0,T);(C^s (\Omega))^*)$ for any $s>0$.
\item[(iv)] If the following energy convergence assumption holds:
\begin{equation}\label{eq:EA}
\lim_{n\to\infty} \int_0^T \J_{\eps_n}( \rho^{\eps_n} (t)) \, dt = \int_0^T \J_0(\rho(t))\, dt
\end{equation}
then we have
\begin{equation}\label{eq:weakp}
\int_0^\infty \int_\Omega \rho v\cdot \xi -p \, \div\xi\, dx \, dt= -\gamma  \int_\Omega \left[\div \xi - \nu\otimes\nu:D\xi\right] |\na \rho|
\end{equation}
for any vector field  $\xi \in C^\infty_c((0,\infty)\times\overline \Omega ; \R^d)$ such that $\xi \cdot n=0$ on $\pa\Omega$.
In other words, $E(t)$  is a weak solution of the Hele-Shaw free boundary problem with surface tension \eqref{eq:HS}-\eqref{eq:CA} with $\gamma_0 = \gamma \rho_c$ and initial condition $E_{in}$.
\end{theorem}

The notion of weak solution of \eqref{eq:HS} that we recover with his theorem is similar to the definition considered for example in \cite{JKM,KMW2,Laux1} and calls for a few classical  comments: 
\begin{enumerate}
\item The continuity equation \eqref{eq:contweakHS} encodes the incompressibility condition $\div v=0$ in $E(t)$ and the free boundary condition $V=-v\cdot \nu$.
\item In \eqref{eq:weakp}, $\nu = \frac{\na \rho}{|\na \rho|}$ stands for the $L^\infty$ density of $\na \rho$ with respect to the total variation $|\na \rho|$ (which exists by Radon-Nikodym's differentiation theorem).
In particular, the term 
$\left( \nu\otimes \nu :D\xi\right) |\na \rho|$ is of the form
$ f(x,\lambda /|\lambda |)d|\lambda|$ with $f$ continuous and $1$-homogeneous and $\lambda =\na \rho$. The integral in  \eqref{eq:weakp} thus makes sense (see for example \cite{serrin}).
Since $\rho=\chi_{E(t)} \in \mathrm{BV}$, we can also write
$$\int_0^\infty \int_\Omega  \rho v\cdot \xi -p \, \div\xi\, dx \, dt= -\gamma_0\int_0^\infty \int_{\pa^* E(t)}  \left[\div \xi - \nu\otimes\nu:D\xi\right] d\H^{d-1} dt$$
where $\pa^*E(t)$ denotes the reduced boundary of $E(t)$. In that case $\nu:\pa^* E(t) \to \mathbb S^{d-1}$  is the measure theoretic unit normal vector.
\item By taking test functions $\xi$ supported in either $E(t)$  or $\Omega\setminus E(t)$, we see that
 \eqref{eq:weakp} implies $\na p=0$ in $\{\rho(t)=0\}$ and $v = -\na p$ in $E(t)$. 
 For general test functions  $\xi$, and taking into account the right hand side of  \eqref{eq:weakp} we further get 
 the surface tension condition $[p]  =  \gamma_0 \kappa $ on $\pa E(t)$ and the normal contact angle condition.
  This can be seen by using the classical formula (for a smooth interface $\Sigma$):
\begin{equation}\label{eq:mc}
\int_\Sigma \div \xi - \nu\otimes\nu : D\xi  d\H^{d-1}  = \int_\Sigma \kappa\, \xi\cdot\nu d\H^{d-1} + \int_\Gamma \tilde n \cdot \xi d\H^{d-2} 
\end{equation}
where $\nu$ is the normal vector to $\Sigma$,  $\kappa$ denotes the mean curvature of $\Sigma$ and $\tilde n$ is the conormal vector along $\Gamma = \pa \Sigma$ (orthogonal to $\pa (\Sigma)$ and tangential to $\Sigma$.
This last term must vanish in \eqref{eq:weakp}, which gives the condition  $\tilde n = n$ which means that $\pa E$ and $\pa \Omega$ must meet orthogonally.

 \item We could remove the pressure $p$ from this definition by taking divergence free vector field in \eqref{eq:weakp}. This approach actually makes sense from a variational point of view since such vector fields are the natural directions for the first order variations of the energy which preserve the mass. We choose this definition to be closer to the formulation \eqref{eq:HS}, and so that one recognizes, in the right hand side of   \eqref{eq:weakp}, the classical weak formulation of the mean-curvature operator.
\end{enumerate}

\subsection{Setting the constants to $1$}
While it was important to state the results with the constants $\sigma$ and $\beta$, it will be less cumbersome to make them equal to $1$ in the proofs.
For that, we note that if we introduce
$$
\bar \phi = \sigma \phi, \quad \bar \eps = \frac\eps{\sqrt{\sigma}} , \quad \bar f'(\rho) = \frac \sigma \beta f'(\rho)
$$  
and rescale the time variable by a factor $\frac{\sigma^{3/2}}{\beta}$, we are led to the system
$$
\begin{cases}
\pa_{\bar t} \rho - \div(\rho \na \bar f'(\rho)) + \div(\rho \na\bar \phi)=0 ,\\
-\bar  \eps^2  \Delta \bar \phi  =  \rho - \bar \phi  \end{cases}
$$
We can thus prove the result when $\sigma=\beta=1$ and recover the mains results from there.

\medskip

\medskip

\medskip

\section{$\Gamma$-convergence: Proof of Theorem \ref{thm:Gamma}}\label{sec:Gamma}
\subsection{Preliminaries}
When $\sigma=\beta=1$  the energy functional is given by
 (see \eqref{eq:JW}) 
 \begin{equation}\label{eq:JW1}
\J_\eps(\rho) = \frac 1 \eps \int_\Omega W(\rho) +  \frac1  2 (\rho-\phi^\eps)^2\, dx +\eps \int_\Omega \frac1  2 |\na \phi^\eps|^2\, dx, \quad  
 W(\rho) =   f(\rho) + a\rho - \frac{1}{2 } \rho^2 
\end{equation}
where $\rho\in L^2(\Omega)$ satisfies  $\rho\geq 0$ and $\int_\Omega \rho\, dx =1$, and
 $\phi^\eps$ solves  \eqref{eq:phi0} with $\sigma=1$, that is
\begin{equation}\label{eq:phi1}
\begin{cases}
 \phi^\eps -\eps^2\Delta  \phi^\eps = \rho  &  \mbox{ in } \Omega\\
 \nabla \phi^\eps \cdot n = 0 & \mbox{ on } \pa\Omega.
\end{cases}
\end{equation}

\medskip

An important role in the proof of Theorem \ref{thm:Gamma} is played by the function
\begin{equation}\label{eq:g}
 g(s)  = \inf_{\rho\geq 0}  \left[ W(\rho) +   \frac 1 2 (\rho- s)^2\right]
\end{equation}
(which is defined for all $s\in \R$). 
In particular, we clearly have 
\begin{equation}\label{eq:lowerbound}
\J_\eps(\rho) \geq  \frac 1 \eps \int_\Omega g(\phi^\eps)  \, dx  +\eps \int_\Omega \frac1  2 |\na \phi^\eps|^2\, dx \qquad \mbox{ where $\phi^\eps$ solves \eqref{eq:phi0}}.
\end{equation}
Our assumptions on $W$, {\bf (H1)}, {\bf (H2)} readily imply the following lemma:
\begin{lemma}\label{lem:g1}
The function $g:\R\to \R$ defined by \eqref{eq:g} satisfies
\begin{equation}\label{eq:gpos} g(s) \geq 0 , \quad \forall s\in \R, \qquad   g^{-1}(0) =  \{0,\rho_c\}
\end{equation}
and
\begin{equation}\label{eq:gmin}
 g(s)\leq \frac 1 2   s^2,  \quad g(s) \leq \frac 1 2 (s-\rho_c)^2.
\end{equation}
\end{lemma}
\begin{proof}
The fact that $W(\rho) \geq 0$, and $W^{-1}(0)=  \{0,\rho_c\}$ yields \eqref{eq:gpos}.
Furthermore, we have $g(s) \leq W(\rho) + \frac 1 2 (\rho-s)^2$ for all $\rho\geq 0$ and we get \eqref{eq:gmin} by taking $\rho=0$ and $\rho=\rho_c$ in this inequality. 
\end{proof}

Lemma \ref{lem:g1}  shows in particular that $g$ is a double-well potential.  The right-hand side of \eqref{eq:lowerbound} is thus the classical Modica-Mortola functional  
\begin{equation}\label{eq:F}
\F_\eps(\psi ) 
=
\begin{cases}
\displaystyle    \frac 1 \eps \int_\Omega g(\psi)  \, dx  +\eps \int_\Omega \frac1  2 |\na \psi|^2\, dx & \mbox{ if } \psi \in H^1(\Omega), \mbox{ and } \int_\Omega \psi\, dx =1 \\
\infty & \mbox{ otherwise}. 
 \end{cases}
\end{equation}
The $\Gamma$-convergence of $\F_\eps $ toward  the perimeter functional has been known since the work on Modica and Mortola \cite{MM77},  Modica \cite{M87} and Sternberg \cite{S88} when $g(s)\sim |s|^p$ for large $|s|$, $p\geq2$. The same result with weaker assumptions on the double-well potential ($g(s) \geq C|s|$ for large $|s|$) was proved by Fonseca and Tartar \cite{FT89}.
The proof  of the liminf property of Theorem \ref{thm:Gamma}, which is presented in Section \ref{sec:liminf}, relies on these results and inequality~\eqref{eq:lowerbound}, which gives:
\begin{equation}\label{eq:JF}
\J_\eps(\rho) \geq \F_\eps(\phi^\eps)\qquad \mbox{ where $\phi^\eps$ solves \eqref{eq:phi0}}.
\end{equation}
\medskip
 
The proof of the limsup property, given in Section \ref{sec:limsup} is more delicate and will require the introduction of the functional 
$$ \G_\eps(\eta,\psi) =  \frac 1 \eps \int_\Omega W(\eta) + \frac1  2 (\eta-\psi)^2 +\eps \int_\Omega \frac1  2 |\na \psi|^2\, dx$$
for all functions $\eta\in L^2(\Omega) $ and $\psi\in H^1(\Omega)$. This is similar to $\J_\eps(\eta)$, but we no longer assume that the potential $\psi$ and the density $\eta$ are related via  \eqref{eq:phi1}.
The key observation is that the minimizer of $\psi\mapsto \G_\eps(\eta,\psi) $ for a fixed $\eta$, naturally solves     the elliptic equation with Neumann boundary conditions  \eqref{eq:phi1} so  that we can recover $\J_\eps(\eta)$. 
More precisely, we will need the following proposition which we prove below:
\begin{proposition}\label{prop:GJ}
The following equalities hold:
\begin{equation}\label{eq:minF} 
\min_{\eta\in L^2(\Omega), \eta\geq 0} \G_\eps(\eta,\phi) = \F_\eps(\phi)\qquad \mbox{ for all  $\phi \in H^1(\Omega)$ with $\int_\Omega \phi \, dx =1$}.
\end{equation}
and 
\begin{equation}\label{eq:minJ} 
 \min_{\psi\in H^1(\Omega)}  \G_\eps(\rho,\psi) = \J_\eps(\rho) \qquad\mbox{ for all $\rho\in L^2(\Omega)$, with $\rho\geq 0$ and  $\int_\Omega \rho\, dx=1$.} 
 \end{equation}
\end{proposition}

In order to prove this proposition, we need to establish a few additional properties of the function $\rho\mapsto g(\rho)$ defined by \eqref{eq:g}:
\begin{lemma}\label{lem:g2}
We have $g\in C^1(\R)\cap C^2_{loc}(\R\setminus\{a\})$ and
\begin{equation}\label{eq:gnu}
g(s)= \frac 1 2 s^2 - f^*(s-a)
\end{equation}
where $f^*$ is the Legendre transform of the convex function $f$ (which we extend to $\R$ by setting $f(\rho)=+\infty$ for $\rho<0$).
Furthermore, for all $s\geq 0$ we have
\begin{equation}\label{eq:eqg} 
g(s) = W(\rho) + \frac 1 2 (\rho-s)^2\quad \Longleftrightarrow \quad s\in \pa f (\rho) +a \quad \Longleftrightarrow \quad \rho = s-g'(s).
\end{equation} 
\end{lemma}
\medskip

\noindent{\bf The Legendre transform:}
We recall that  the Legendre transform of the convex function $f$ is defined by 
$$ f^*(s)= \sup_{\rho\in \R} \left( \rho s - f(\rho)\right) .$$
This function $f^*$ is convex and since $f\geq 0$, $f(0)=0$ and $f(\rho)=+\infty$ for $\rho<0$ we easily check that 
$$ f^*(s)=0 \qquad \mbox{ for all } s\leq 0.$$
Furthermore, ({\bf H2}) implies that $f(\rho)\geq \frac 1 2 \rho^2-a\rho$ for all $\rho\geq 0$, and so  $f^*:\R\to\R$ is continuous and satisfies $f^*(s)\leq  \frac 1 2 (s+a)^2$.
We also have the classical property of the subdifferentials
$$ s \in \pa f (\rho) \Leftrightarrow \rho \in \pa f^*(s)$$
where $\pa f(\rho) = f'(\rho)$ if $\rho>0$ and $\pa f(0) = (-\infty,0]$. 
We thus have
\begin{equation}\label{eq:*'} {f^*}' (s) = 0 \quad\mbox{ for } s\leq 0, \quad {f^*}'(s) = (f')^{-1}(s)>0  \quad\mbox{ for } s> 0 \end{equation}
and ({\bf H1}) now implies that $f^* \in C^1(\R)\cap C^2_{loc}(\R\setminus\{ 0\})$.
Finally, ({\bf H1}) also implies $f'(\rho) \geq \delta \rho -1$ for $\rho\geq \delta^{-1}$ which gives
\begin{equation}\label{eq:f*'}
 {f^*}'(s) \leq  \delta^{-1} (1+s) \qquad\mbox{ for all } s\in \R.
 \end{equation}

\medskip

We now turn to the proof of Proposition  \ref{prop:GJ} (the proof of Lemma \ref{lem:g2} will be presented right after):
\begin{proof}[Proof of Proposition \ref{prop:GJ}]
Given $\phi \in H^1(\Omega)$ with $\int_\Omega \phi \, dx =1$,
\eqref{eq:g}  implies 
$$\G_\eps(\eta,\phi) \geq  \F_\eps(\phi) \qquad \mbox{ for all $\eta\in L^2(\Omega)$, with $\eta\geq 0$}.$$
Furthermore, if we  define the function $\rho_\phi(x)$ by
$\rho_\phi(x): = \phi (x)- g'(\phi(x))$ for all $x\in \Omega$, then \eqref{eq:eqg}  implies
$ g(\phi(x)) = W(\rho_\phi(x)) + \frac 1 2 (\rho_\phi(x)-\phi(x))^2$
and so
$$\G_\eps(\rho_\phi,\phi) =  \F_\eps(\phi).$$
This implies  \eqref{eq:minF} if we can show that $\rho_\phi\in L^2$ and $\rho_\phi\geq 0$:
Using, \eqref{eq:gnu} we  get  the alternative formula  
$ \rho_\phi(x) = {f^*}' (\phi(x)-a)$ which, together with 
 \eqref{eq:*'} and  \eqref{eq:f*'},  implies that $\rho_\phi(x)\geq 0$ and 
 $\rho_\phi\in L^2(\Omega)$.
 Equality \eqref{eq:minF}  thus  follows.
\medskip

In order to prove \eqref{eq:minJ}, we note that for a given $\rho\in L^2(\Omega)$, we have
$$
 \min_{\psi\in H^1(\Omega)}  \G_\eps(\rho,\psi) 
=  \frac 1 \eps \int_\Omega W(\rho) \, dx + \min_{\psi\in H^1(\Omega)} \left(  \frac 1 \eps \int_\Omega  \frac1  2 (\rho-\psi)^2 +\eps \int_\Omega \frac1  2 |\na \psi|^2\, dx\right)
$$
This latest minimization problem is classical and the minimum is reached for the  function $\phi^\eps\in H^1(\Omega)$ which satisfies the weak formulation \eqref{eq:phiweak} of  \eqref{eq:phi1}.
We deduce that
$$ \min_{\psi\in H^1(\Omega)}  \G_\eps(\rho,\psi) = \G_\eps(\rho,\phi^\eps)
$$
and since $\phi^\eps$ is the solution of \eqref{eq:phi1} we have $\G_\eps(\rho,\phi^\eps)=\J_\eps(\rho)$ (when $\rho\geq 0$ and $\int \rho=1$). Equality \eqref{eq:minJ} follows.
\end{proof}
\medskip

\begin{proof}[Proof of Lemma \ref{lem:g2}]
Recalling the fact that  $ W(\rho) =   f(\rho) + a\rho - \frac{1}{2 } \rho^2 $,
we can write:
\begin{align}
g(s) 
 &=  \inf_{\rho\geq 0}  \left[  f(\rho) +a\rho - \frac {1}{ 2}  \rho^2  +  \frac {1}{ 2}  (\rho-s)^2\right] \nonumber \\
 & =  \inf_{\rho\geq 0}  \left[  f(\rho) - \rho ( s-a)   +  \frac {1 }{ 2}   s^2\right] \nonumber \\
& =  \frac {1 }{ 2}   s^2 +  \inf_{\rho\geq 0}   \left(f (\rho) -   \rho (  s-a)\right)  \nonumber \\
& =  \frac {1 }{ 2}  s^2 -\sup_{\rho\in \R}   \left(     \rho (  s-a)-f(\rho)\right) \label{eq:gsup}
\end{align}
(where we recall that $f(\rho)=+\infty$ for $\rho<0$)
which gives \eqref{eq:gnu}.
The properties of the Legendre transform discussed above imply that $g\in C^1(\R) \cap C^2_{loc} (\R\setminus\{a\})$.
\medskip

Given $s\geq0$, we have
$ g(s) = W(\rho) + \frac 1 2 (\rho-s)^2 $ if and only if the infimum in the definition  of $g$, \eqref{eq:g}, is reached for  $\rho$. This infimum is reached for the same $\rho$ as the supremum 
in \eqref{eq:gsup}, that is for all $\rho $ such that $s-a \in \pa f(\rho)$ (recall that $f$ is convex).
This gives the first equivalence.
Furthermore, basic properties of the  Legendre transform imply
$$ 
s-a \in \pa f (\rho)   \Longleftrightarrow  \rho \in \pa f^*(s-a) 
$$
It remains to see that \eqref{eq:gnu} gives $\pa f^*(s-a)  = s-g'(s)$ to get the second equivalence.
\end{proof}

\subsection{The liminf property (Theorem \ref{thm:Gamma}-(i))}\label{sec:liminf}
The proof of the liminf property ((Theorem \ref{thm:Gamma}-(i)) follows from \eqref{eq:JF} and the usual argument for Modica-Mortola's convergence for $\F_\eps$. Since this argument is relatively simple, we recall it here for convenience:
We introduce the function  $F$ such that 
\begin{equation}\label{eq:FGHJ} F'(s)  = 
\sqrt{2g(s)}\, , \quad s\in(0,\rho_c), \quad F(0)=0
\end{equation}
which we extend by constant to $\R_+$ by setting $F(s)= \int_0^{\rho_c} \sqrt{2g(s)}\, ds = \gamma \rho_c$ for $s\geq \rho_c$.
We point out that \eqref{eq:gmin} implies
$$ F'(s)\leq \sqrt{2g(s)} \leq \min\{ s, \rho_c-s\}_+\qquad \forall s\geq 0$$
so that $F$ is a Lipschitz function.

Given $\rho\in L^2(\Omega)$, $\rho\geq 0$ and $\phi^\eps \in H^1$ the solution of \eqref{eq:phi1} (which satisfies in particular $\phi^\eps\geq 0$ by the maximum principle), we have
$$ \int_\Omega |\na F(\phi^\eps)| \leq  \int_\Omega \sqrt{2 g(\phi^\eps)} |\na \phi^\eps| \leq  \frac 1 \eps \int_\Omega g(\phi^\eps)  \, dx  +\eps \int_\Omega \frac1  2 |\na \phi^\eps|^2\, dx.
$$
Together with  \eqref{eq:lowerbound}, this implies
\begin{equation}\label{eq:BVbound} 
\int_\Omega |\na F(\phi^\eps)| 
 \leq 
\J_\eps(\rho).
\end{equation}
We now consider a sequence $\rho^\eps$ which converges to $\rho$ in $L^1$.
If $\liminf_{\eps\to 0} \J_\eps (\rho^\eps)=\infty$, then the result is obviously true, so we assume that $\liminf_{\eps\to 0} \J_\eps (\rho^\eps)<\infty$ and consider a subsequence $\rho^{\eps_n}$ such that $\liminf_{\eps\to 0} \J_\eps (\rho^\eps) = \lim  \J_{\eps_n} (\rho^{\eps_n})$ and $ \J_{\eps_n} (\rho^{\eps_n})\leq C$.
In particular, we have $\rho^{\eps_n}\geq 0$ and $\int_\Omega \rho^{\eps_n}\, dx =1$ and so $\int_\Omega \rho(x)\, dx =1$.
Up to another subsequence, we can further assume that $\rho^{\eps_n}\to\rho$ a.e. in $\Omega$.
Formula \eqref{eq:JW} then gives
$$ \int_\Omega W(\rho^{\eps_n})  \, dx +  \int_\Omega\frac1  2 (\rho^{\eps_n}-\phi^{\eps_n})^2 \leq C {\eps_n}.
 $$
This implies that $\phi^{\eps_n}\to \rho$ in $L^1$ (and a.e. up to yet another subsequence) and
 since $W(\rho^{\eps_n}) \to W(\rho)$ a.e., Fatou's lemma implies that
$W(\rho)=0$ a.e., that is  $\rho = \rho_c \chi_{E}$ for some set $E\subset \Omega$.

Since $\rho\mapsto F(\rho)$ is continuous,  $F(\phi^{\eps_n})$ converges a.e. and in $L^1(\Omega)$  to $F(\rho_c \chi_{E}) = F(\rho_c)\chi_E = \gamma \rho_c\chi_E = \gamma\rho$ (we can use Lebesgue dominated convergence theorem since $F(\phi^{\eps_n}) \leq \gamma \rho_c$ in $\Omega$).
Finally, \eqref{eq:BVbound} and the lower semicontinuity of the BV norm implies 
$$
\liminf  \J_{\eps_n} (\rho^{\eps_n}) \geq
\liminf   \int_\Omega |\na F(\phi^{\eps_n})| 
\geq  \int_\Omega |\na (\gamma \chi_E)| =  \int_\Omega |\na (\gamma \rho)| = \J_0(\rho).
$$
This concludes the proof of the liminf property.

\subsection{The Limsup property (Theorem \ref{thm:Gamma}-(ii))}\label{sec:limsup}
First, we recall the  classical result for the Modica-Mortola functional $\F_\eps$:
\begin{theorem}[\cite{S88,FT89,Leoni}]\label{thm:GF}
Assume that $\Omega$ is an open subset of $\R^d$ with Lipschitz boundary and that $g$ is a continuous double-well potential.
Then for every $\rho\in L^1(\Omega)$ 
there exists a sequence $\psi^\eps \in L^1(\Omega)$ such that $\psi^\eps \to \rho$ in $L^1(\Omega)$ (and a.e.) and 
$$ \limsup_{\eps\to 0^+} \F_{\eps} (\psi^\eps) \leq \J_0(\rho)$$
\end{theorem}

Classically, the liminf property for $\F_\eps$ is proved requiring some growth conditions on $W$ at $\pm\infty$, but such conditions are not required for the limsup property. We also recall that
when $0\leq \rho\leq \rho_c$ a.e. in $\Omega$, the classical construction of the sequence $\psi^\eps$ guarantees that
$$0\leq \psi^\eps (x)\leq \rho_c \qquad \mbox{ a.e. in } \Omega.$$
\medskip

\begin{proof}[Proof of Theorem \ref{thm:Gamma}-(ii)]
If $\J_0(\rho)=\infty$, the construction of the recovery sequence is trivial, so we  can  assume that 
$\rho\in BV(\Omega; \{0,\rho_c\})$ and $\int_\Omega \rho(x)\, dx=1$. We can thus write
$$ \rho = \rho_c \chi_E, \quad \mbox{ with } \quad \rho_c \int _\Omega \chi_E(x)\, dx =1.$$
As noted in Remark \ref{rem:sup}, for such a function $\rho$ we have
$$ 
\lim_{\eps\to 0} \J_\eps( \rho_c \chi_E )  =  \frac { \rho_c } {4} \int_\Omega |\na \rho| > \gamma H^{n-1}(\pa^* E ) =\J_0(\rho)$$
so we cannot take $\rho^\eps = \rho$ as a recovery sequence. 
Instead, we use 
Theorem \ref{thm:GF}, which implies the existence 
of a sequence $\psi^\eps\in H^1(\Omega) $ such that $\psi^\eps\to \rho $ in $L^1(\Omega)$ and 
$$ 
\limsup_{\eps\to0^+} \F_{\eps}(\psi^\eps) \leq \J_0(\rho).
$$
This sequence $\psi^\eps$ satisfies in particular the mass constraint $\int_\Omega \psi^\eps\, dx=1$, but in the proof of Theorem~\ref{thm:GF} (see for instance \cite{Leoni}), 
one actually constructs a family of sequences $\psi^\eps_t$ with $t\in[0,1]$ such that 
\begin{equation}\label{eq:fbhkj}
\psi^\eps_t\to \rho  \mbox{ in } L^1(\Omega),\quad \mbox{ and } 
\quad
\limsup_{\eps \to0^+}  \frac 1 {\eps} \int_\Omega g(\psi^\eps_t)  \, dx  +{\eps} \int_\Omega \frac1  2 |\na\psi^\eps_t|^2\, dx  \leq \J_0(\rho) \qquad \forall t\in [0,1]
\end{equation}
with $\psi^\eps_0\leq \rho_c \chi_E$,  $\psi^\eps_1 \geq \rho_c \chi_E$ and $t\mapsto\psi^\eps_t$ continuous.
The classical recovery sequence given by Theorem~\ref{thm:GF}  is then obtained by choosing $t_\eps$  so that  $\int_\Omega\psi^\eps_{t_\eps}\, dx= 1$.
For our proof, we will make a slightly different choice of $t_\eps$:
We claim that one can choose $t_\eps\in[0,1]$ so that 
\begin{equation}\label{eq:intphi}
 \int_\Omega \psi^\eps_{t_\eps} - g'(\psi^\eps_{t_\eps})\,dx = 1.
 \end{equation}
Indeed, we can write (see \eqref{eq:gnu}) $s - g'(s)={f^*}'(s-a)$ which is a non decreasing function of $s$ 
(see \eqref{eq:*'})
and so 
$$ \psi^\eps_{0} - g'( \psi^\eps_{0}) \leq \rho_c \chi_E - g'(\rho_c \chi_E) = \rho_c \chi_E,  \qquad  \psi^\eps_{1} - g'( \psi^\eps_{1}) \geq \rho_c \chi_E - g'(\rho_c \chi_E) = \rho_c \chi_E$$
(where we used the fact that $g$ is $C^1$ and has minimums at $0$ and $\rho_c$ and so $g'(\rho_c\chi_E) = 0 $ a.e.).
A continuity argument implies that there exists $t_\eps\in [0,1]$ such that \eqref{eq:intphi} holds.
We then denote $\psi^\eps = \psi^\eps_{t_\eps}$. We recall that $\psi^\eps\to \rho$ in $L^1$ and   a.e. in $\Omega$.

\medskip

We now define  the function $\rho^\eps$ by
$$\rho^\eps(x) = \psi^\eps(x) - g'(\psi^\eps(x)) = {f^*}'(\psi^\eps(x)-a).$$
Using \eqref{eq:*'}, \eqref{eq:f*'} and \eqref{eq:intphi}, we get
$$ \rho^\eps\geq0, \quad \rho^\eps\in L^2(\Omega), \quad \int_\Omega \rho^\eps\, dx=1.$$
Next, we note that $\rho^\eps$ converges  to $ \rho - g'(\rho) = \rho $ a.e.  (since $g'(0)=g'(\rho_c)=0$).
Since $\psi^\eps \leq \rho_c$, the monotonicity of ${f^*}'$ implies
$\rho^\eps\leq {f^*}'(\rho_c-a)$; we can thus use Lebesgue's dominated convergence theorem to get that $\rho^\eps \to \rho $ in $L^1(\Omega)$. 
\medskip

For this choice of $\rho^\eps$, \eqref{eq:eqg}  implies $g(\psi^\eps) = W(g^\eps) + \frac 1 2 (\rho^\eps-\psi^\eps)$ and so 
$$  \frac 1 {\eps} \int_\Omega g(\psi^\eps)  \, dx  +{\eps} \int_\Omega \frac1  2 |\na \psi^\eps|^2\, dx
 = \G_{\eps} (\rho^\eps,\psi^\eps).$$
In particular, \eqref{eq:fbhkj} gives
$$
 \limsup_{{\eps}\to0}\G_{\eps}(\rho^\eps,\psi^\eps)  
 =  \limsup_{{\eps}\to0} \frac 1 {\eps} \int_\Omega g(\psi^\eps)  \, dx  +{\eps} \int_\Omega \frac1  2 |\na \psi^\eps|^2\, dx
 \leq \J_0(\rho).
 $$
Finally,  \eqref{eq:minJ} implies
$$ \J_\eps(\rho^\eps ) \leq \G_{\eps}(\rho^\eps,\psi^\eps)$$ 
(importantly, this inequality requires $\int_\Omega\rho^\eps(x)\, dx = 1$ but not  $\int_\Omega \psi^\eps \, dx = 1$)
and so 
$$ \limsup_{\eps \to0}\J_{\eps}(\rho^\eps ) \leq  \limsup_{{\eps}\to0}\G_{\eps}(\rho^\eps,\psi^\eps)  \leq \J_0(\rho).$$
\end{proof}

\medskip

\medskip

\medskip

\section{Proof of Theorem \ref{thm:conv1} - part 1}\label{sec:lim1}
In this section, we prove the first part of Theorem \ref{thm:conv1}, namely 
the strong convergence of the density $\rho^{\eps_n} $ in $L^\infty(0,T;L^1(\Omega))$ to $\rho(x,t)$ satisfying  the continuity equation  \eqref{eq:weak11} and the energy dissipation property \eqref{eq:energy2}. We point out that this part does not require the energy convergence assumption \eqref{eq:EA}.
Below, we denote $\rho^\eps$ instead of $\rho^{\eps_n}$ to simplify the notations.

We can write the $\eps$ continuity equation  \eqref{eq:weak11}  as
$$\pa_t \rho^\eps + \div j^\eps =0$$
with  $j^\eps = \rho^\eps v^\eps$.
We then have the following a priori estimates:
\begin{lemma}\label{lem:unifestimates}
Let $\rho_{in}^\eps(x)$ be such that $\J_\eps(\rho^\eps_{in}) \leq M $
and  $\rho^{\eps} (x,t)$  be a solution 
of \eqref{eq:weak}-\eqref{eq:phi0} given by Theorem~\ref{thm:existence}.  
The followings hold:
\item[(i)]  $\J_{\eps}(\rho^\eps(t)) \leq M$ for all $t\geq 0$ and  $\| j^{\eps}\|_{L^2(0,\infty; L^1(\Omega))}\leq M $.
\item[(ii)]  If $W$ satisfies \eqref{eq:Wqg}  (that is if $f(\rho)$ satisfies (H3)), then there exists a constant $C(M)$ such that
 $\| \rho^\eps\| _{L^\infty(0,\infty;L^2(\Omega) )}\leq C(M)$,  $\| j^{\eps}\|_{L^{2}(0,\infty;L^{4/3}(\Omega))}\leq C(M) $ and 
$\| \rho^{\eps}(t)-\rho^{\eps}(s)\| _{W^{-1,4/3}(\Omega)} \leq C(M)\sqrt{t-s}$
 for any $0\leq s\leq t$.
\end{lemma}

\begin{proof}
Inequality \eqref{eq:diss} gives
\begin{equation}\label{eq:JM}
\J_{\eps} (\rho^\eps(T)) +  \int_0^T\int_\Omega \frac{|j^\eps|^2}{\rho^\eps} \, dx\, dt \leq \J_\eps(\rho_{in})
\end{equation}
which immediately gives the first bound in (i). Furthermore, we have
$$  \int_\Omega  |j^\eps|\, dx \leq 
\left( \int_\Omega \rho^\eps\, dx \right)^{1/2}\left(   \int_\Omega \frac{|j^\eps|^2}{\rho^\eps} \, dx \right)^{1/2}
\leq \left(   \int_\Omega \frac{|j^\eps|^2}{\rho^\eps} \, dx \right)^{1/2}
$$
and so \eqref{eq:JM} implies
$$
\int_0^T\left(  \int_\Omega  |j^\eps|\, dx\right)^2\, dt \leq M.
$$

\medskip

To prove (ii), we note that, 
in view of \eqref{eq:JW1}, we also have
\begin{equation}\label{eq:Wbound}
\int_\Omega W(\rho^\eps(x,t))\, dx \leq \eps \J_{\eps} (\rho^\eps(t)) \leq \eps M
\end{equation}
and so \eqref{eq:Wqg} implies
the bound on $\rho^\eps$ in $L^\infty(0,T;L^2(\Omega) )$ depending only on $M$ (for $\eps\leq1$).
We then have:
$$  \int_\Omega  |j^\eps|^{4/3}\, dx \leq 
\left( \int_\Omega |\rho^\eps|^2\, dx \right)^{1/3}\left(   \int_\Omega \frac{|j^\eps|^2}{\rho^\eps} \, dx \right)^{2/3}
$$
and so \eqref{eq:JM} implies
$$
\int_0^T\left(  \int_\Omega  |j^\eps|^{4/3}\, dx\right)^{3/2}\, dt \leq M
$$

Finally, for a given test function $\psi \in H^1(\Omega)$, the continuity equation \eqref{eq:weak11} implies
$$
\int_\Omega \rho^\eps(x,t) \psi(x)\, dx - \int_\Omega \rho^\eps(x,s) \psi(x)\, dx = \int_s^t \int_\Omega j^\eps  \cdot\na \psi\, dx\, d\tau
$$ 
and so  
\begin{align*}
\left| 
\int_\Omega \big(\rho^\eps(x,t)- \rho^\eps (x,s) \big)\psi(x)\, dx \right|
& \leq \left(\int_s^t\left(  \int_\Omega  |j^\eps|^{4/3}\, dx\right)^{3/2}\, dt\right)^{1/2}\left(\int_s^t\left( \int_\Omega  |\na \psi|^4 \, dx\right)^{1/2}\, d\tau \right)^{1/2}\\
& \leq  C(M) \| \psi\|_{W^{1,4}(\Omega)}\left(t-s\right)^{1/2}
\end{align*}
and  (ii) follows.
\end{proof}

The main result of this section is the following proposition: 
\begin{proposition}\label{prop:rhostrong} 
 Let $\rho_{in}^\eps(x)$ be such that $\J_\eps(\rho^\eps_{in}) \leq M $
 and $\rho^\eps(x,t)$  be a  solution 
of \eqref{eq:weak}-\eqref{eq:phi0} given by Theorem \ref{thm:existence}. 
Consider a sequence such that $\eps_n\to 0$. The followings hold:
\item[(i)] There exists a subsequence (still denoted $\eps_n$) along which $\rho^{\eps_n}(t)$ converges  to $\rho(t)$ in $W^{-1,4/3}(\Omega)$ locally uniformly with respect to $t$   and $j^{\eps_n}$ converges to $j$ weakly in $ L^{2}(0,\infty;L^{4/3}(\Omega)) $.
\item[(ii)] 
There exists $v\in (L^2(\Omega\times(0,\infty),d\rho))^d$ such that $j=\rho v$ and the following continuity equation holds
\begin{equation}\label{eq:cont22}
\begin{cases}
 \pa_t \rho +\div \rho v =0,\\ 
 \rho(x,0)=\rho_{in}(x).
 \end{cases}
 \end{equation}
\item[(iii)] Up to another subsequence, $\rho^{\eps_n}(t)$ converges to $\rho(t)$ strongly in $L^1(\Omega)$, locally uniformly in $t$. Furthermore, for all $t>0$ we have 
$$\rho(t)\in BV(\Omega;\{0,\rho_c\}).$$
\end{proposition}
Note that (i) and (ii) are classical. The most important statement is thus (iii) which proves in particular that we have phase separation   in the limit $\eps\to0$.

\begin{proof}
The a priori estimates of Lemma \ref{lem:unifestimates} give (i).
Furthermore, we can pass to the limit 
 in \eqref{eq:weak11} to get
 $$
\int_\Omega \rho_{in} (x) \zeta(x,0)\, dx + \int_0^\infty \int_\Omega \rho\, \pa_t\zeta +j \cdot \na \zeta \, dx = 0 
$$
for any function $\zeta\in C^\infty_c([0,\infty)\times\overline \Omega)$. 
This is the continuity equation \eqref{eq:cont22} if we can show that $j$ can be written in the form $\rho v$.
We prove that by using an argument that can be found, for example, in \cite{MRS}:
For a scalar measure $\mu$ and a vectorial measure $F$, we define the function
$$ \Theta: (\mu,F)\mapsto 
\begin{cases}
\displaystyle \int_0^T\int_\Omega \frac {|F|^2}{\mu}  & \mbox{  if } F \ll \mu \mbox{ a.e. } t\in [0,T] ;\\
+\infty & \mbox{ otherwise}.
\end{cases}
$$
This function $\Theta$ is  lower semi-continuous for the weak convergence of measure (see \cite{AFP}, Theorem 2.34).
Together with the uniform bound $\Theta (\rho^{\eps_n} ,j^{\eps_n} ) = \int_0^T\int_\Omega \rho ^{\eps_n} |v^{\eps_n}|^2 \leq C$ (see Lemma \ref{lem:unifestimates} (ii)), it implies that $j$ is absolutely continuous with respect to $\rho$ and that there exists $v(t,\cdot) \in L^2 (d\rho(t))$ such that $j = \rho v$.

\medskip

We must now prove the strong convergence of $\rho^{\eps_n}$.
We use the function  $F$ introduced in the proof of the liminf property (see \eqref{eq:FGHJ}). We recall that $F$ is a Lipschitz function such that
$F(0)=0$ and $ F'(s)  = \sqrt{2g(s)}$ for  $s\in(0,\rho_c)$ and that we then have (see \eqref{eq:BVbound}):
\begin{equation} \label{eq:BVphi}
  \int_\Omega |\na F(  {\phi^\eps})| \, dx\leq  \J_\eps (\rho^\eps) 
\end{equation}
where $\phi^\eps$ is the solution of \eqref{eq:phi1}.
The boundedness of the energy $\J_{\eps_n}(\rho^{\eps_n }) $ thus implies some a priori estimates for the auxiliary function
$$ \psi^n :=  F( \phi^{\eps_n }).$$
More precisely, \eqref{eq:BVphi} and Lemma \ref{lem:unifestimates} (i) imply that
\begin{equation}\label{eq:psiBV}
  \psi^n \mbox{  is bounded in } L^\infty((0,\infty);BV(\Omega)) .
\end{equation}
In order to get the strong convergence of $\psi^n$ in $x$ and $t$, we must get some uniform convergence in $t$ (weakly in space) and use a Lions-Aubin type Lemma to conclude. For this, we write
\begin{align}
\psi^n 
& = [ F(\phi^{\eps_n }) - F (\rho^{\eps_n }) ] +[ F(\rho^{\eps_n })-\gamma \rho^{\eps_n}]  +\gamma \rho^{\eps_n} \label{eq:psin}
\end{align}
and we are going to show that the first two terms in the right hand side go to zero (uniformly in $t$):
\begin{itemize}
\item Formula \eqref{eq:JW1} and the energy bound (Lemma \ref{lem:unifestimates} (i))  imply
$$ \| \rho^{\eps_n}(t)- \phi^{\eps_n}(t) \|_{L^2(\Omega)}^2 \leq2  \eps_n \J_\eps(\rho^{\eps_n}(t))\leq 2  \eps_n \J_\eps(\rho_{in})\leq C\eps_n$$
and since $F$ is Lipschitz, we deduce
\begin{equation}\label{eq:ghj1}
 \| F(\rho^{\eps_n}(t))- F( \phi^{\eps_n}(t))\|_{L^2(\Omega)}^2 
\leq C  \| \rho^{\eps_n}(t)-  \phi^{\eps_n}(t)\|_{L^2(\Omega)}^2  \leq C\eps_n \qquad \forall t>0.
\end{equation}
\item For the second term, we note that the function $\rho\mapsto F(\rho)-\gamma \rho$ vanishes when $\rho=0 $ and $\rho=\rho_c$, which are also the zeroes of $\rho\mapsto W(\rho)$.
So given $V_\delta$  a $\delta$ neighborhood of $\{0,\rho_c\}$ in $\R_+$,
the continuity of $F$ and $W$ together with \eqref{eq:Wqg} implies
$$ 
| F(\rho) -\gamma \rho|^2 \leq C_\delta W(\rho) \mbox{ for } \rho \notin V_\delta
$$ 
for some constant $C_\delta$ (more precisely,  \eqref{eq:Wqg} and the fact that $F$ is bounded implies this inequality for large $\rho$, the continuity together with a compactness argument then give the inequality for all $\rho \notin V_\delta$).
Since $F$ is Lipschitz, we also have
$$ | F(\rho)-\gamma \rho | \leq C\delta \mbox{ for } \rho\in V_\delta.$$
Formula \eqref{eq:JW1} and the energy bound (see \eqref{eq:Wbound}) then imply
\begin{align*}
\int_\Omega   |F(\rho^{\eps_n}(t)) -\gamma \rho^{\eps_n}(t)|^2 \, dx 
& \leq \int_{\{\rho^{\eps_n}(t)\in V_\delta\}}   C\delta^2 \, dx  + C_\delta \int_{\{\rho^{\eps_n}(t)\notin V_\delta\}}  W(\rho^{\eps_n}(t))\, dx\\
& \leq C|\Omega|  \delta^2 + C _\delta \eps_n   \J_\eps(\rho^{\eps_n}(t))\\
& \leq C|\Omega|  \delta^2 + C _\delta \eps_n    .
\end{align*}
We deduce
$$\limsup_{n\to\infty} \| F(\rho^{\eps_n} ) -\gamma \rho^{\eps_n}\|_{L^\infty(0,\infty;L^2(\Omega))} \leq C \delta. 
$$
Since this holds for all $\delta>0$, we must have
\begin{equation}\label{eq:ghj2}
\limsup_{n\to\infty} \| F(\rho^{\eps_n} ) -\gamma \rho^{\eps_n}\|_{L^\infty(0,\infty;L^2(\Omega))} =0.
\end{equation}
\end{itemize}
Finally, we already know that $  \rho^{\eps_n}$ converges locally uniformly in $t$, with respect to the  $W^{-1,4/3}(\Omega)$ norm, to $\rho$, we deduce from \eqref{eq:psin}, \eqref{eq:ghj1} and \eqref{eq:ghj2} that 
\begin{equation}\label{eq:psiH}
\psi^n(t) \to \rho(t) \mbox{ in $W^{-1,4/3}(\Omega)$, uniformly with respect to $t\in[0,T]$}
\end{equation}
(for any $T>0$).
\medskip

We now recall the following Lions-Aubin compactness type result (the proof of which is identical to that of Lemma B.1 in \cite{KMW2}):
\begin{lemma}\label{lem:LA}
Let $u_n$ be a sequence of function bounded in $ L^\infty(0,T ;L^1( \Omega))$ such that
$ u_n$ is bounded in $L^\infty((0,T);\BV(\Omega))$ and $u_n\to u$ in $L^\infty((0,T);W^{-1,4/3}(\Omega))$.
Then
$$ \sup_{t\in[0,T]} \|u_n(t)-u(t)\|_{L^1 (\Omega)} \to 0.$$
\end{lemma}
Lemma \ref{lem:LA} together with \eqref{eq:psiBV} and \eqref{eq:psiH} imply
$$ \psi^n \to \rho \qquad \mbox{ strongly in }  L^\infty((0,T);L^1( \Omega)).$$
In particular, the lower semicontinuity of the $BV$ norm and \eqref{eq:psiBV}  imply that $\rho \in L^\infty((0,T);BV( \Omega))$.
\medskip

Finally, using \eqref{eq:psin} together with 
\eqref{eq:ghj1} and \eqref{eq:ghj2} we see that $\psi^n -\gamma \rho^{\eps_n}$ converges to zero strongly in $L^\infty(0,T;L^2(\Omega))$
so the strong convergence of $\psi^n$ also implies that
$$\rho^{\eps_n} \to \rho \qquad \mbox{ strongly in }  L^\infty((0,T);L^1( \Omega)).$$
\medskip

It remains to show that $\rho(t)$ is of the form $\rho_c\chi_{E(t)}$ for all $t$: We note that given $t>0$, we can extract a subsequence which converges a.e. in $\Omega$ so Formula \eqref{eq:JW1} and the energy bound  (see \eqref{eq:Wbound})  implies
 $$\int_\Omega W(\rho(t))\, dx = 0$$ 
hence  $\rho (x,t)\in \{0,\rho_c\}$ a.e. $x\in  \Omega$ (for all $t>0$) and so $\rho = \rho_c \chi_{\{\rho>0\}}$ which completes the proof  of Proposition \ref{prop:rhostrong}. 
\end{proof}

\medskip

\medskip

\medskip

\section{Proof of Theorem \ref{thm:conv1} - Part 2}

\subsection{First variation of the energy}
We proved in Theorem \ref{thm:Gamma} that the energy functional $\J_\eps$ $\Gamma$-converges to $\J_0$.
We now want to prove that if 
$\rho^\eps$ is a sequence of densities that converge to $\rho$ and if 
 $\J_\eps(\rho^\eps)$ converges to $\J_0(\rho)$, then we  have the convergence of the first variations. This is the statement of the next proposition which plays a crucial role in the proof of the second part of  Theorem \ref{thm:conv1}.

\begin{proposition}\label{prop:firstvar}
Given a sequence of functions $\rho^\eps\in L^1(\Omega)$  and $\phi^\eps$ the corresponding solution of \eqref{eq:phi0}. 
If $\rho^\eps\to\rho$ strongly in $L^1(\Omega)$ and 
\begin{equation}\label{eq:cvass}
\lim_{\eps\to 0} \J_\eps (\rho^\eps) =\J_0 (\rho),
 \end{equation}
then for all $\xi \in C^1(\Omega , \R^d)$ satisfying $\xi\cdot n=0$ on $\pa\Omega$, we have:
\begin{align}
&\lim_{\eps\to 0}
  \eps^{-1}  \int_\Omega 
 [f(\rho^\eps) +a\rho^\eps - \rho^\eps \phi^\eps ] \div \xi-  \rho^\eps \na \phi^\eps\cdot \xi  \, dx
   = \gamma \int_\Omega \left[  \div \xi - \nu \otimes \nu :D\xi\right] |\na \rho|
 \label{eq:limitdJ} 
\end{align}
where $f = \frac{\na \rho}{|\na \rho|}$ and $n$ denotes the outward normal unit vector to the fixed boundary $\pa \Omega $.
\end{proposition}
If $\rho^\eps$ is smooth enough, we can rewrite the left hand side of \eqref{eq:limitdJ}  as
$$
 -  \eps^{-1}   \int_\Omega 
 (f'(\rho^\eps)+a - \phi^\eps) \na \rho^\eps \cdot \xi\, dx
$$
Since we have $\J_\eps(\rho) = \frac 1 \eps \int_\Omega f(\rho)+a\rho  - \frac 1 2 \rho \phi^\eps\, dx$, this  is the first variation of $\frac{d}{ds} \J_\eps(\rho_s) |_{s=0} $ along the perturbation 
\begin{equation}\label{eq:rhos1}
\begin{cases}
\pa_s \rho_s + \na \rho_s \cdot \xi =0 \\
\rho_s|_{s=0} = \rho.
\end{cases}
\end{equation}
It might seem strange to the reader that we are using a perturbation which does not preserve the mass (unless we require that $\div \xi=0$). But this perturbation preserves characteristic functions, and is thus the appropriate perturbation when we are interested in the regime $\eps\ll1$.

Similar results have been proved  for  different energy functionals.
In particular, a classical result of Reshetnyak \cite{Reshetnyak} gives that if  $\chi_{E_\eps}$ converges to $\chi_E$ strongly in $L^1$, then the convergence of the perimeter 
$$ P(E_\eps):=\int|\na \chi_{E_\eps}| \to \int|\na \chi_E|,$$
implies the convergence of the first variation to
$$ \int(\div \xi - \nu \otimes\nu:D\xi) |D\chi_E|.$$
We refer to \cite{LM,LO,JKM} for similar results for other approximations of the perimeter functional.

Proceeding as in \cite{LO}, Proposition \ref{prop:firstvar} also implies the following time-dependent version:
\begin{corollary}\label{cor:bhlj}
Given a sequence of functions $\rho^\eps\in L^\infty(0,T;L^1(\Omega))$   such that $\rho^\eps\to\rho$ strongly in $L^\infty(0,T;L^1(\Omega))$ and $\phi^\eps$ the corresponding solution of \eqref{eq:phi0}.
If 
\begin{equation*}
\lim_{\eps\to 0} \int_0^T \J_\eps (\rho^\eps(t))\, dt = \int_0^T\J_0 (\rho(t))\, dt,
 \end{equation*}
then we have, for all $\xi \in C^1(\Omega , \R^d)$ satisfying $\xi\cdot n=0$ on $\pa\Omega$,
\begin{align*}
&\lim_{\eps\to 0}
  \eps^{-1} \int_0^T \int_\Omega 
 [f(\rho^\eps) +a\rho^\eps - \rho^\eps \phi^\eps ] \div \xi-  \rho^\eps \na \phi^\eps\cdot \xi  \, dx\, dt
   =\int_0^T  \int_\Omega \left[  \div \xi - \nu \otimes \nu :D\xi\right] |\na \rho|
\end{align*}
where $f = \frac{\na \rho}{|\na \rho|}$ and $n$ denotes the outward normal unit vector to the fixed boundary $\pa \Omega $.
\end{corollary}
\medskip

A key tool in the proof of Proposition \ref{prop:firstvar} is the following lemma:
\begin{lemma}\label{lem:first}
Given $\rho(x)$,  $\phi^\eps(x)$ solution of \eqref{eq:phi1} and 
for all $\xi \in C^1(\Omega ; \R^d)$ satisfying $\xi\cdot n=0$ on $\pa\Omega$, we have
\begin{align*}
\frac{1}{ \eps}    \int_\Omega 
 [f(\rho ) +a\rho  - \rho \phi^\eps] \div \xi-  \rho \na \phi^\eps \cdot \xi  \, dx 
& =\frac{1}{ \eps}  \int_\Omega [ W(\rho)   + \frac1 2 ( \rho- \phi^\eps)^2 ] \div \xi \, dx \\
& \quad +  \frac \eps{2} \int_\Omega |\na \phi^\eps|^2 \div \xi \, dx-  \eps\int_\Omega \na \phi^\eps\otimes \na \phi^\eps:D\xi\, dx.
\end{align*}
\end{lemma}
\begin{proof}
The definition of $W(\rho)$ \eqref{eq:W} gives:
\begin{align*}
\frac{1}{ \eps}    \int_\Omega 
 [f(\rho ) +a\rho  - \rho \phi^\eps] \div \xi-  \rho \na \phi^\eps \cdot \xi  \, dx 
 & = \frac{1}{ \eps}  \int_\Omega [  W(\rho)   + \frac1 2  \rho^2  - \rho  \phi^\eps ] \div \xi - \rho \na \phi^\eps \cdot\xi\, dx \\
 & = \frac{1}{ \eps}  \int_\Omega [  W(\rho)   + \frac1 2 ( \rho-\phi^\eps)^2 ] \div \xi  -\frac 1 2   |\phi^\eps|^2 \div \xi  - \rho \na \phi^\eps \cdot\xi \, dx.
 \end{align*}
Next, we write (using the fact that $\rho - \phi^\eps= - \eps^2 \Delta \phi^\eps$):
 \begin{align*}
 \int_\Omega\frac 1 2 |\phi^\eps|^2 \div \xi +  \rho \na \phi^\eps \cdot \xi  \, dx
 & = \int_\Omega (\rho-\phi^\eps) \na \phi^\eps \cdot \xi \, dx\\
& = -\frac {\eps^2} 2 \int_\Omega |\na \phi^\eps|^2 \div \xi + \eps^2  \int_\Omega \na \phi^\eps\otimes \na \phi^\eps:D\xi \, dx.
\end{align*}
The lemma follows.
\end{proof}

\begin{proof}[Proof of Proposition \ref{prop:firstvar}]
In view of Lemma \ref{lem:first}, it is enough to show, under the assumptions of Proposition \ref{prop:firstvar}, that
\begin{equation}\label{eq:f1}
\lim_{\eps\to 0} \frac{1}{ \eps}  \int_\Omega [ W(\rho^\eps)   + \frac1 2 ( \rho^\eps- \phi^\eps)^2 ] \div \xi\, dx  +  \frac \eps{2 } \int_\Omega |\na \phi^\eps|^2 \div \xi \, dx=\gamma \int_\Omega \div \xi |\na \rho|
\end{equation}
and 
\begin{equation}\label{eq:f2}
\lim_{\eps\to 0} \eps\int_\Omega \na \phi^\eps\otimes \na \phi^\eps :D\xi \, dx =\gamma \int_\Omega \nu \otimes \nu :D\xi|\na \rho|.
\end{equation}

 The proof largely follows the proof of Proposition 5.2 in \cite{KMW2} and is given here for the sake of completeness.
  As in the proof of Theorem \ref{thm:Gamma}-(i) (liminf property), we use the function $F(s)$ defined by
$$  F(0)=0  ,\quad F'(s) = \sqrt{2g(s)}  \quad \forall s\in(0,\rho_c),\quad  F(s)=F(\rho_c) \mbox{ for } s \geq \rho_c.$$
We also introduce the functions $u_\eps$ and $v_\eps$ defined by
$$
  u_\eps^2 : =  \eps^{-1} [W(\rho^\eps)   + \frac1 2 ( \rho^\eps- \phi^\eps)^2 ]  \quad  \mbox{ and } \quad  v_\eps ^2=\frac{\eps}{2} |\na {\phi^\eps}|^2 .
$$
We have in particular (see \eqref{eq:JW1}):
$$\J_\eps (\rho^\eps)  = \int_\Omega  u_\eps^2 + v_\eps^2\, dx $$
and the definition of the function $g$ (see \eqref{eq:g}) implies $ \eps^{-1} g(\phi^\eps)  \leq u_\eps^2$ and so
\begin{equation}\label{eq:young}
|\na F(\phi^\eps)| \leq\sqrt{ 2 g(\phi^\eps)}  |\na {\phi^\eps}| \leq 2  u_\eps v_\eps =  u_\eps^2 + v_\eps^2 - (u_\eps-v_\eps)^2
\end{equation}

Next, the strong convergence of $\rho^\eps$ and \eqref{eq:ghj1} imply that $F(\phi^\eps) $ converges to $F(\rho)$ strongly in $L^1$. Since $\rho$ is a characteristic function and $F(0)=0$, $F(\rho_c) =\gamma \rho_c$, we deduce $F(\phi^\eps) \to F(\rho)= \gamma \rho$ strongly  in $L^1$. The lower semi-continuity of the BV norm implies:
$$ \liminf_{\eps\to 0} \int_\Omega |\na F(\phi^\eps)| \, dx \geq \gamma \int_\Omega |\na\rho|.$$
On the other hand,  the convergence assumption \eqref{eq:cvass} implies 
$$ \int_\Omega  u_\eps^2 + v_\eps^2\, dx =  \J_\eps (\rho^\eps) \to \gamma \int_\Omega |\na\rho|.$$
Inequality \eqref{eq:young} thus implies:
\begin{align}
& u_\eps^2 + v_\eps^2 - |\na F(\phi^\eps)| \to 0 \quad \mbox{ in } L^1(\Omega) \label{eq:conv2} \\
&  \int_\Omega   |\na F(\phi^\eps)| \, dx  \to \gamma \int_\Omega |\na\rho| \label{eq:conv1}\\
& u_\eps-v_\eps \to 0 \quad \mbox{ in } L^2(\Omega). \label{eq:conv3}
\end{align}
These facts  allow us to establish the limit \eqref{eq:f1}.
Indeed, using first the definition of $u^\eps$ and $v^\eps$, then the limit \eqref{eq:conv2} and finally  \eqref{eq:conv1} (together with Proposition \ref{prop:BV}), we can write:
\begin{align}
\lim_{\eps\to 0} 
\frac{1}{ \eps}  \int [ W(\rho^\eps)   + \frac1 2 ( \rho^\eps- \phi^\eps)^2 ] \div \xi  +  \frac \eps{2 } \int |\na \phi^\eps|^2 \div \xi 
&  =\lim_{\eps\to 0}  \int_\Omega ( u_\eps^2 + v_\eps^2) \div \xi \, dx \nonumber\\
&  =\lim_{\eps\to 0}  \int_\Omega |\na F(\phi^\eps)|  \div \xi \, dx\nonumber \\
 &= \gamma  \int_\Omega \div \xi |\na\rho| .\nonumber  
\end{align}
Furthermore, \eqref{eq:conv2} and \eqref{eq:conv3} yields:
\begin{equation}\label{eq:conv4}
2u_\eps^2 - |\na F(\phi^\eps)| \to 0 \quad \mbox{ in } L^1(\Omega) 
\end{equation}
which we use to pass to the limit in the term involving $\na {\phi^\eps}\otimes\na\phi^\eps$. Indeed, we can write
$$
2\eps \int_\Omega \pa_i {\phi^\eps} \pa_j {\phi^\eps} \pa_i\xi_j \, dx
 = 2\eps \int_\Omega \frac{\pa_i {\phi^\eps}}{|\na {\phi^\eps}|}\frac{ \pa_j {\phi^\eps} }{|\na {\phi^\eps}|} \pa_i\xi_j \, {|\na {\phi^\eps}|^2}dx
 =  \int_\Omega \frac{\pa_i {\phi^\eps}}{|\na {\phi^\eps}|}\frac{ \pa_j {\phi^\eps} }{|\na {\phi^\eps}|} \pa_i\xi_j \, 2u_\eps^2 dx
$$ 
and since $\frac{\pa_i {\phi^\eps}}{|\na {\phi^\eps}|}\frac{ \pa_j {\phi^\eps} }{|\na {\phi^\eps}|} \pa_i\xi_j $ is bounded in $L^\infty$, \eqref{eq:conv4} implies that
$$
\lim_{\eps\to 0} 
2\eps \int_\Omega \pa_i {\phi^\eps} \pa_j {\phi^\eps} \pa_i\xi_j \, dx
=
\lim_{\eps\to 0} 
 \int_\Omega \frac{\pa_i {\phi^\eps}}{|\na {\phi^\eps}|}\frac{ \pa_j {\phi^\eps} }{|\na {\phi^\eps}|} \pa_i\xi_j \, |\na F(\phi^\eps)|  dx.
$$
Using the fact that $F'(\phi)\geq 0$, we can also write
$$
\lim_{\eps\to 0} 
2\eps \int_\Omega \pa_i {\phi^\eps} \pa_j {\phi^\eps} \pa_i\xi_j \, dx
=
\lim_{\eps\to 0} 
 \int_\Omega \frac{\pa_i F( {\phi^\eps})}{|\na F( {\phi^\eps})|}\frac{ \pa_j F(\phi^\eps) }{|\na F(\phi^\eps)|} \pa_i\xi_j \, |\na F(\phi^\eps)|  dx.
$$
and using \eqref{eq:conv1} and Proposition \ref{prop:measure} we deduce
\begin{equation}\label{eq:lim2}
\lim_{\eps\to 0} 
2\eps \int_\Omega \pa_i {\phi^\eps} \pa_j {\phi^\eps} \pa_i\xi_j \, dx
=\gamma  \int_\Omega \frac{\pa_i \rho }{|\na \rho|}\frac{ \pa_j \rho }{|\na \rho |} \pa_i\xi_j \, |\na \rho|  
\end{equation}
which is \eqref{eq:f2}.

\end{proof}

\subsection{The pressure equation}
We now complete the proof of   Theorem \ref{thm:conv1}:
We use the result of the previous section to pass to the limit in 
\eqref{eq:weak12}, which we rewrite here using the notations of Proposition~\ref{prop:rhostrong}:
\begin{equation}\label{eq:frnj}
\int_0^\infty \int_\Omega j^{\eps_n} \cdot \xi - \eps_n^{-1}\rho^{\eps_n}\na \phi^{\eps_n}  \cdot \xi - \eps_n^{-1} [ \rho^{\eps_n} f'(\rho^{\eps_n} ) -f(\rho^{\eps_n})]\, \div\xi\, dx \, dt= 0
\end{equation}
where  $\xi \in C^\infty_c((0,\infty)\times\overline \Omega ; \R^d)$ is such that $\xi \cdot n=0$ on $\pa\Omega$.

We now introduce the function
\begin{align*}
p^{\eps_n}   & := \eps_n^{-1} \rho^{\eps_n}  [  f'(\rho^{\eps_n} ) + a  - \phi^{\eps_n} ] + m^{\eps_n}(t)  
\end{align*}
where $m^{\eps_n}(t)$, independent of $x$, is chosen so that 
$$
\int_\Omega p^{\eps_n}(x,t)\, dx = 0 \qquad \forall t>0.
$$
With this notation, we can
rewrite \eqref{eq:frnj} as
\begin{equation}\label{eq:bfk}
\int_0^\infty \int_\Omega j^{\eps_n}  \cdot \xi \, dx\, dt 
= 
 -  \eps_n^{-1} \int_0^\infty \int_\Omega [f(\rho^{\eps_n} )+a  \rho^{\eps_n}  - \rho^{\eps_n} \phi^{\eps_n} ]  \div \xi 
 -\rho^{\eps_n} \na \phi^{\eps_n} \cdot \xi 
 \, dx\, dt  +  \int_0^\infty \int_\Omega p^{\eps_n}   \, \div\xi\,  dx \, dt
\end{equation}
where we recognize the middle term as the first variation of the energy appearing in \eqref{eq:limitdJ}. 

We note that Lemma \ref{lem:first} implies 
$$
\left| \eps_n^{-1}  \int_\Omega [f(\rho^{\eps_n} )+a  \rho^{\eps_n}  - \rho^{\eps_n} \phi^{\eps_n} ]  \div \xi 
 -\rho^{\eps_n} \na \phi^{\eps_n} \cdot \xi \, dx\, \right|
\leq C \| D\xi\|_{L^\infty(\Omega)} \J_{\eps_n}  (\rho^{\eps_n})
$$
so \eqref{eq:bfk} together with Lemma \ref{lem:unifestimates} - (i) give
\begin{align*}
\left| \int_0^T \int_\Omega p^{\eps_n}   \, \div\xi\,  dx \, dt \right| & \leq C  \| D\xi\|_{L^1(0,T ; L^\infty( \Omega))}  + C \| j^\eps\|_{L^2(0,T;L^1( \Omega))} \| \xi\| _{L^2(0,T;L^\infty( \Omega))} \\
&\leq C(T)  \| \xi\| _{L^2(0,T;C^1(\Omega))} .
\end{align*}
This implies that
 $\na p^{\eps_n} $  is bounded in $L^2((0,T); (C^1(\Omega))^*)$ and (proceeding as in \cite{KMW2}) that
 $p^{\eps_n}  $ is uniformly bounded in $L^2((0,T); (C^s(\Omega))^*)$ and has a weak-* limit $p(x,t)$.

Finally, using Corollary \ref{cor:bhlj}, we can  pass to the limit in \eqref{eq:bfk} to get
$$
\int_0^\infty \int_\Omega j\cdot \xi \, dx\, dt 
= 
  - \gamma \int_0^\infty \int_\Omega \left[  \div \xi - \nu \otimes \nu :D\xi\right] |\na \rho|   + \int_0^\infty \int_\Omega p  \, \div\xi  \, dx\, dt 
$$
which is \eqref{eq:weakp} and thus complete the proof of Theorem \ref{thm:conv1}.

\begin{remark}
Using the fact that $W'(\rho_c) = f'(\rho_c) - \rho_c + a =0$ (since $\rho_c$ is a minimum of $W$), we can also write the pressure $p^{\eps_n} $ as follows:
$$
p^{\eps_n}  =  \eps_n^{-1} \rho^{\eps_n}  [  f'(\rho^{\eps_n} ) -f'(\rho_c)+ \rho_c  - \phi^{\eps_n} ] + m^{\eps_n}(t) .
$$
In this form, it is a bit easier to see why this function should converge despite the factor $\eps_n^{-1}$, since $  \rho^{\eps_n}  [  f'(\rho^{\eps_n} ) -f'(\rho_c)+ \rho_c  - \phi^{\eps_n} ] $ converges to 
$\rho   [  f'(\rho  ) -f'(\rho_c)+ \rho_c  - \phi ] $ which vanishes when $\phi=\rho = \rho_c\chi_{E}$.
However, proving that $p^{\eps_n}$ is bounded directly, and obtaining a stronger notion of convergence than $L^2((0,T); (C^s(\Omega))^*)$, does not seem realistic. Indeed, the properties of the limiting pressure $p(x,t)$ are related to that of the mean curvature of $\pa E(t)$. Showing that $p(x,t)$ is in some $L^p$ space would thus be equivalent to establishing some regularity properties for the solution of the limiting Hele-Shaw flow with surface tension, a very delicate problem.
\end{remark}

\medskip

\appendix

\section{A couple of facts about $BV$ functions}
We recall here a couple of important results about $\BV$ functions which we used in our proof (we refer the reader  to \cite{AFP}   for details). 
First we have
\begin{proposition}\label{prop:BV}
Let $f_k$ be a sequence of functions such that $f_k \to f$ in $L^1(\Omega)$ when $k\to\infty$.
Then
$$\liminf_{k\to\infty} \int_\Omega \zeta(x)  |\na f_k|
\geq  \int_\Omega \zeta(x) |\na f|  
$$
for all $\zeta \in C(\Omega)$ with $\zeta \geq 0$.
Furthermore, if 
$\int_\Omega |\na f_k| \to \int_\Omega |\na f|,$
then
$$
\lim_{k\to\infty} \int_\Omega\zeta (x) |\na f_k|    
= \int_\Omega\zeta(x)  |\na f| \qquad 
\mbox{ for all $\zeta\in C(\Omega)$}.$$
\end{proposition}
 We also used the  following particular case of Reshetnyak's continuity theorem (see \cite{AFP} Theorem 2.39)
\begin{proposition}\label{prop:measure}
Let $f_k$ be a sequence of function such that $f_k \to f$ in $L^1(\Omega)$ and $\int_\Omega |\na f_k|\, dx \to \int_\Omega |\na f|$.
Then
$$
\lim_{k\to \infty}\int_\Omega  \zeta(x) \frac{\pa_i f_k}{|\na f_k|}\frac{\pa_jf_k}{|\na f_k|} |\na f_k| =
 \int_\Omega \zeta(x) \frac{\pa_i f}{|\na f|}\frac{\pa_jf}{|\na f|} |\na f|
 $$
for all $\zeta\in C(\overline \Omega)$.
\end{proposition}

\bibliographystyle{plain}
\bibliography{chemotaxis-soft}

\end{document}